\documentclass[11pt,oneside]{amsart}
\usepackage{fouriernc} 

\usepackage{gjotex/Environments}   
\usepackage{gjotex/Definitions}    
\usepackage{gjotex/PageSetup}      

\ifdraft{

\usepackage{gjotex/DraftSetup}

\usepackage[nonotes]{refdeps/refdeps}

} 
{}

\ifthenelse{\isundefined{\proofof}}{
  \newcommand{\proofof}[1]{}
  \newcommand{\writeref}[1]{}
  \newcommand{\marginref}[1]{}
  \newcommand{\marginproof}[2]{}
}{}

\newcommand{\idop}{\mathbb{1}}
\newcommand{\can}{\ka} 
\newcommand{\wncan}{\wt{\can}} 
\newcommand{\CAN}{\mathfrak{c}} 

\newcommand{\Anglearrow}[1]{\rotatebox{#1}{$\Rightarrow$}}

\hypersetup{
 pdfkeywords={stable homotopy hypothesis, Picard 2-category},
 pdfauthor={Nick Gurski,Niles Johnson,Angélica Osorno},
}

\theoremstyle{plain}
\newtheorem*{shh}{Stable Homotopy Hypothesis}


\usepackage{gjotex/AuthorInfo}     

\title{The 2-dimensional stable homotopy hypothesis}
\authorGurski
\authorJohnson
\authorOsorno
\thanks{The third author was supported by Simons Foundation Grant No. 359449,
the Woodrow Wilson Career Enhancement Fellowship and NSF grant
DMS-1709302.}

\begin{document}

\begin{abstract}
  We prove that the homotopy theory of Picard 2-categories is
  equivalent to that of stable 2-types.
\end{abstract}

\maketitle


\section{Introduction}

Grothendieck's Homotopy Hypothesis posits an equivalence of homotopy
theories between homotopy $n$-types and weak $n$-groupoids.  We pursue
a similar vision in the stable setting.  Inspiration for a stable
version of the Homotopy Hypothesis begins with
\cite{Seg74Categories,may72geo} which show, for 1-categories, that
symmetric monoidal structures give rise to infinite loop space
structures on their classifying spaces.  Thomason
\cite{Tho95Symmetric} proved this is an equivalence of homotopy
theories, relative to stable homotopy equivalences.  This suggests
that the categorical counterpart to stabilization is the presence of a
symmetric monoidal structure with all cells invertible, an intuition
that is reinforced by a panoply of results from the group-completion
theorem of May \cite{May1974Einfty} to the Baez-Dolan stabilization
hypothesis \cite{BD1995Higher,Bat2017Operadic} and beyond.  A stable
homotopy $n$-type is a spectrum with nontrivial homotopy groups only
in dimensions $0$ through $n$.  The corresponding symmetric monoidal
$n$-categories with invertible cells are known as Picard
$n$-categories.  We can thus formulate the Stable Homotopy Hypothesis.

\begin{shh}
  There is an equivalence of homotopy theories between
  $\Pic^n$, Picard $n$-categories equipped with categorical
  equivalences, and $\Sp_0^n$, stable homotopy $n$-types
  equipped with stable equivalences.
\end{shh}

For $n=0$, the Stable Homotopy Hypothesis is the equivalence of
homotopy theories between abelian groups and Eilenberg-Mac Lane
spectra.  The case $n=1$ is described by the second two authors in
\cite{JO12Modeling}.  Beyond proving the equivalence of the homotopy
theories, they constructed a dictionary in which the algebraic
invariants of the stable homotopy 1-type (the two homotopy groups and
the unique $k$-invariant) can be read directly from the Picard
category. Moreover, they gave a construction of the stable 1-type of
the sphere spectrum.

The main result of this paper, \cref{thm:main}, is the Stable Homotopy
Hypothesis for $n = 2$.  In this case, the categorical equivalences
are biequivalences.  The advantage of being able to work with
categorical equivalences is that the maps in the homotopy category
between two stable 2-types modeled by strict Picard 2-categories are
realized by symmetric monoidal pseudofunctors---not general
zigzags. In fact, the set of homotopy classes of maps between two
strict Picard 2-categories $\cD$ and $\cD'$ is the quotient of the set
of symmetric monoidal pseudofunctors $\cD\to \cD'$ modulo monoidal
pseudonatural equivalence.

In future work we will develop 2-categorical models for the 2-type of
the sphere and for fiber/cofiber sequences of stable 2-types.  We can
apply these to give algebraic expressions for the secondary operations
arising from a stable Postnikov tower and for the low-dimensional
algebraic $K$-groups of a commutative ring.  Moreover, via the theory
of cofibers (cokernels) associated with a Postnikov tower, we may shed
new light on the theory of symmetric monoidal tricategories.

Our proof of the 2-dimensional stable homotopy hypothesis is a
culmination of previous work in \cite{GJO2017KTheory} and
\cite{GJOS2017Postnikov}.  Although we have attempted to make the
current account as self-contained as possible, we rely heavily
on this and other previous work.  We include selective reviews as needed.  The proof of
the main theorem functions as an executive summary of the paper, and
the reader may find it helpful to begin reading there.

\subsection{The homotopy hypothesis and categorical stability}

The Stable Homotopy Hypothesis (SHH) in dimension 2 establishes an
equivalence between items (\ref{it:st2}) and (\ref{it:pic2}) of the
following conjecturally equivalent homotopy theories:
\begin{enumerate}
\item\label{it:top46} 3-connected topological 6-types and topological weak equivalences;
\item\label{it:cat46} weak 6-groupoids with only a single cell in dimensions 0 to
  3 and 6-categorical equivalences;
\item\label{it:st2} stable topological 2-types and stable weak
  equivalences; and
\item\label{it:pic2} Picard 2-categories and 2-categorical equivalences.
\end{enumerate}
The equivalence of homotopy theories between items (\ref{it:top46})
and (\ref{it:st2}) is an immediate consequence of the Freudenthal
Suspension Theorem (FST).  Equivalences between (\ref{it:top46}) and
(\ref{it:cat46}) or between (\ref{it:cat46}) and (\ref{it:pic2}) are
the subjects of, respectively, the Unstable Homotopy Hypothesis (UHH) and
the Baez-Dolan Stabilization Hypothesis (BDSH).  Both of these have been
studied in lower dimensions, but nothing approaching dimension 6 for
fully algebraic notions of higher category has appeared as of this
writing.  The following schematic diagram sketches these statements.

\ 
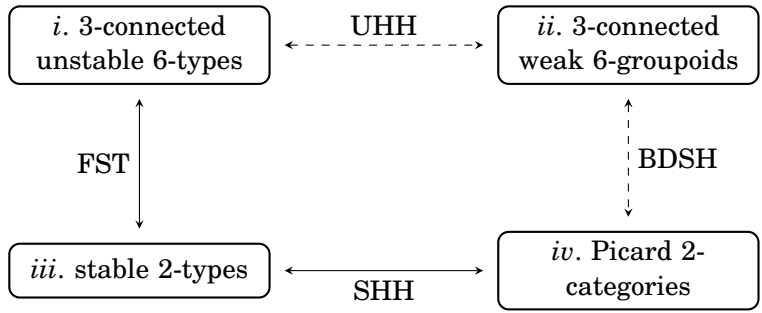
\begin{center}
\begin{tikzpicture}[x=65mm, y=30mm,
  block/.style ={rectangle, thick,
    text width=9em, align=center, rounded corners,
    minimum height=2em, outer sep=2mm},
  stable/.style={draw=black},
  unstable/.style={draw=black}]
  \draw
  (0,0) node[unstable, block] (i) {\textit{i}.\ 3-connected unstable 6-types}
  (1,0) node[unstable, block] (ii) {\textit{ii}.\ 3-connected weak 6-groupoids}
  (0,-1) node[stable, block] (iii) {\textit{iii}.\ stable 2-types}
  (1,-1) node[stable, block] (iv) {\textit{iv}.\ Picard 2-categories};
  \draw[<->] (i) -- node[left] {FST} (iii);
  \draw[<->] (iii) -- node[below] {SHH} (iv);
  \draw[<->, dashed] (i) -- node[above] {UHH} (ii);
  \draw[<->, dashed] (ii) -- node[right] {BDSH} (iv);
\end{tikzpicture}
\end{center}
\ 

Analogues of the Baez-Dolan  Stabilization Hypothesis have been proved for
Tamsamani's weak $n$-categories
\cite{Simpson1998Breen,Simpson2012Homotopy}, for certain algebras over
operads in $(\infty,n)$-categories \cite{GH2015Enriched}, and for
Rezk's $\Theta_n$-categories \cite{Bat2017Operadic}.  However, it is
unknown how these homotopical notions of higher categories compare
with fully algebraic weak $n$-categories for $n > 2$ (see
\cite{LP20082nerves} for a comparison in dimension 2).  In
particular, work of Cheng and the first author
\cite{CG2011periodicII,CG2014Iterated} shows clearly that, in
codimension greater than 1, additional subtleties arise even in the
\emph{formulation} of the Stabilization Hypothesis for fully algebraic
notions of weak $n$-category.  In the absence of a comparison between
algebraic and homotopical notions of higher category, using the
homotopical definitions like $\Theta_n$-categories does not address
the core motivation for the (Stable) Homotopy Hypothesis: to compare
fully topological notions with fully algebraic ones.

\subsection{Outline}
We begin with necessary topological background in
\cref{sec:topological-background}, particularly recalling the theory
of group-completion and an elementary consequence of the relative
Hurewicz theorem.  Next we recall the relevant algebra of symmetric
monoidal structures on 2-categories in \cref{sec:algebra}, including a
discussion of both the fully weak case (symmetric monoidal
bicategories) and what might be called the semi-strict case
(permutative Gray-monoids).

The core construction in this paper is a ``Picardification.''  That
is, the construction of a Picard 2-category from a general permutative
Gray-monoid, while retaining the same stable homotopy groups in
dimensions 0, 1, 2.  This entails a group-completion, and to apply
previous work on group-completion we develop an independent theory of
symmetric monoidal bicategories arising from $E_\infty$ algebras in
2-categories in \cref{sec:operads}.  This theory is an extension of
techniques first developed by the first and third authors for the
little $n$-cubes operad \cite{GO2012Infinite}.

Our subsequent analysis in \cref{sec:gp-cpn} uses the
 fundamental 2-groupoid of Moerdijk-Svensson \cite{MS93Algebraic}
(\cref{sec:W}), the $K$-theory for 2-categories developed in
\cite{GJO2017KTheory} (\cref{sec:K-thy}), and the topological
group-completion theorem of May \cite{May1974Einfty}
(\cref{sec:top-gp-cpn}).  We combine these to conclude with the proof of
the main theorem in \cref{sec:main}.

\subsection{Acknowledgements}

The authors would like to thank Peter May, Mikhail Kapranov, and Chris
Schommer-Pries for helpful conversations.

\section{Topological background}
\label{sec:topological-background}

In this section we review basic topological background needed for the
work in this paper. We will use topological spaces built using the
geometric realization of a simplicial nerve for 2-categories, and we
begin by fixing notation and reviewing the monoidal properties of the
relevant functors.  We then turn to group completion and Postnikov
truncations, both of which play a key role in this work.

\subsection{Topological spaces and nerves of 2-categories}

\begin{notn}\label{notn:top}
  For spaces, we work in the category of compactly-generated weak
  Hausdorff spaces and denote this category $\Top$.
\end{notn}

\begin{notn}
  We let $\sset$ denote the category of simplicial sets.
\end{notn}

\begin{notn}
  We let $|-|$ and $S$ denote, respectively, the geometric realization
  and singular functors between simplicial sets and topological
  spaces.
\end{notn}

\begin{notn}\label{notn:cat-iicat}
  We let $\cat$ denote the category of categories and functors, and
  let $\iicat$ denote the category of 2-categories and 2-functors.
  Note that these are both 1-categories.
\end{notn}

The category of 2-categories admits a number of morphism variants, and
it will be useful for us to have separate notations for these.
\begin{notn}
  We let $\iicatps$ denote the category of 2-categories with
  pseudofunctors and let $\iicatnps$ denote the category of
  2-categories with normal pseudofunctors, that is, pseudofunctors which preserve the identities strictly.  We let $\iicatnop$ denote
  the category of 2-categories and normal oplax functors.  Note that
  these are all 1-categories.
\end{notn}

The well-known nerve construction extends to 2-categories (in fact to
general bicategories) in a number of different but equivalent ways
\cite{Gur2009Nerves,CCG10Nerves}.
\begin{notn}
  We let $N$ denote the nerve functor from categories to simplicial
  sets.  By abuse of notation, we also let $N$ denote the
  2-dimensional nerve on $\iicatps$.  This nerve has 2-simplices given
  by 2-cells whose target is a composite of two 1-cells, as in the
  display below.
  
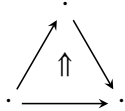
\[
    \begin{tikzpicture}[x=15mm,y=15mm]
      \draw[tikzob,mm] 
      (0:0) node (a) {\cdot}
      (60:1) node (b) {\cdot}
      (0:1) node (c) {\cdot}
      (.5,.35) node {\Uparrow}
      ;
      \path[tikzar,mm] 
      (a) edge node {} (b)
      (b) edge node {} (c)
      (a) edge node {} (c)
      ;
    \end{tikzpicture}
  \]
  This is the nerve used by \cite{MS93Algebraic} in their study of the
  Whitehead 2-groupoid (see \cref{sec:W}).  A detailed study of this
  nerve, together with 9 other nerves for bicategories, appears in
  \cite{CCG10Nerves} with further work in \cite{CHR2015Bicategorical}.
\end{notn}

\begin{prop}\label{prop:classical-monoidal-functors} 
  The functors 
  \begin{itemize}
  \item $|-|\cn \sset \to \Top$
  \item $S\cn \Top \to \sset$
  \item $N\cn \iicat_{ps} \to \sset$
  \end{itemize}
  are strong symmetric monoidal with respect to cartesian product. The
  adjunction between geometric realization and the singular functor
  is monoidal in the sense that the unit and the counit are
  monoidal natural transformations.
\end{prop}
\begin{proof}
  The fact that $S$ and $N$ preserve products follows from the fact
  that they are right adjoints. The statements about $|-|$ and the (co)unit are standard
  (see for example \cite[Section 3.4]{GZ1967Calculus}), but depend on good
  categorical properties of compactly-generated Hausdorff spaces.
\end{proof}

\begin{notn}\label{notn:B}
We will write $B$ for the classifying space of a 2-category, so $B \cC=
|N \cC|$ for a 2-category $\cC$.
\end{notn}

\subsection{Group-completions and Postnikov truncations}
In this section we review and fix terminology for group-completion,
Postnikov truncation, and the attendant notions of equivalence for
spaces.

\begin{defn}\label{defn:gpcompletion}
  A map of homotopy associative and homotopy commutative $H$-spaces
  $f\cn X\to Y$ is a \emph{(topological) group-completion} if
 \begin{itemize}
\item $\pi_0(Y)$ is a group and $f_\ast \cn \pi_0(X) \to \pi_0(Y)$ is
an
   algebraic group-completion, and
 \item for any field of coefficients $k$, the map
 \[
   H_*(X;k)[\pi_0(X)^{-1}] \to H_*(Y;k)
 \]
 induced by $f_*$ is an isomorphism.
 \end{itemize}
\end{defn}

\begin{rmk}
  The group-completion of a given homotopy associative and homotopy
  commutative $H$-space $X$, if one exists, is unique up to weak
  homotopy equivalence by the Whitehead theorem. The definition was
  motivated by the work of Barratt and Barratt-Priddy
  \cite{Bar1961Note,BP1972Homology} and of Quillen
  \cite{Quillen71GpCompletion}, who proved that for a homotopy
  commutative simplicial monoid $M$, the map $M \to \Om BM$ satisfies
  the homology condition of \cref{defn:gpcompletion}.  The work of
  \cite{May1974Einfty,Seg74Categories} constructs group-completions
  for $E_\infty$-spaces, and both of these are foundational for
  results which we use in this paper (see
  \cref{thm:property-of-Kthy,thm:spectra-from-Einfty}).
   \end{rmk}
   
 \begin{notn}
   Let $X$ be a homotopy associative and homotopy commutative
   $H$-space. If a topological group-completion of $X$ exists, we
   denote it by $X\to \Om BX$.
 \end{notn}


\begin{defn}\label{defn:P_n-equiv}
  Let $P_n$ denote the $n$th Postnikov truncation on the category of
  spaces.  This is a localizing functor, and the
  \emph{$P_n$-equivalences} are those maps $f\cn X \to Y$ which induce
  isomorphisms on $\pi_i$ for $0 \leq i \leq n$ and all choices of
  basepoint. We likewise define $P_n$-equivalences for maps of
  simplicial sets.
\end{defn}

We will also require the slightly weaker, and more classical,
notion of $n$-equivalence.
\begin{defn}\label{defn:n-equiv}
  Let $n \ge 0$.  A map of spaces $f\cn X \to Y$ is an
  \emph{$n$-equivalence} if, for all choices of basepoint $x \in X$,
  the induced map
  \[
  \pi_q(X,x) \to \pi_q(Y,f(x))
  \]
  is a bijection for $0 \leq q < n$ and a surjection for $q =
  n$.  Note that this notion does not satisfy the 2-out-of-3
  property in general.
\end{defn}

Clearly every $(n+1)$-equivalence is a $P_n$-equivalence, and
every $P_n$-equivalence can be replaced, via Postnikov truncation, by
a zigzag of $(n+1)$-equivalences.  Indeed, the collection of
$P_n$-equivalences is the closure of the collection of 
$(n+1)$-equivalences with respect to the 2-out-of-3 property.

\begin{rmk}
  For a map of spaces $f\cn X \to Y$, the following are equivalent.
  \begin{itemize}
  \item The map $f$ is an $n$-equivalence.
  \item For all choices of basepoint $x \in X$, the homotopy fiber of
    $f$ over $x$ is an $(n-1)$-connected space.
  \item The pair $(Mf,X)$ is $n$-connected, where $Mf$ denotes the
    mapping cylinder of $f$.
    That is, $\pi_0(X)$ surjects onto
    $\pi_0(Mf)$ and the relative homotopy groups
    $\pi_q(Mf,X)$ are trivial for $0 < q \leq n$.
  \end{itemize}
\end{rmk}

We require the following result connecting $n$-equivalences with
group-completion, particularly the subsequent corollary.  The case
$n = \infty$ follows from Whitehead's theorem, but we have not
discovered a reference for finite $n$.  We give a proof below.
\begin{prop}\label{prop:n-equiv}
  Let $n \ge 0$ and let $f\cn X \to Y$ be an $E_\infty$ map.  If $f$
  is an $n$-equivalence then the group-completion
  $\Om B f\cn \Om B X \to \Om B Y$ is an $n$-equivalence.
\end{prop}

\begin{defn}
  A space $Y$ is said to be \emph{$k$-coconnected} if $\pi_i Y = 0$
  for $i \geq k$.
\end{defn}
\begin{cor}\label{cor:n-equiv-alt} 
  Let $n \ge 0$ and let $f\cn X \to Y$ be an $E_\infty$ map.  If $f$
  is a $P_n$-equivalence and $Y$ is $(n+1)$-coconnected, then
  $\Om B f\cn \Om B X \to \Om B Y$ is a $P_n$-equivalence.
\end{cor}
\marginproof{cor:n-equiv-alt}{prop:n-equiv}

Our proof of \cref{prop:n-equiv} makes use of the relative Hurewicz
theorem, specifically a corollary below which we have also not
discovered in
the literature.
\begin{thm}[Relative Hurewicz {\cite[Theorem 4.37]{Hatcher2002Algebraic}}]
Suppose $(X, A)$ is an $(n − 1)$-connected pair of path-connected
spaces
with $n \ge 2$ and $x_0 \in A$, and suppose that $\pi_1(A,x_0)$ acts
trivially on $\pi_1(X,A,x_0)$.  Then the Hurewicz homomorphism
\[
\pi_i(X, A, x_0) \to H_i(X, A)
\]
is an isomorphism for $i \leq n$.
\end{thm}
\begin{rmk}
  This result can be extended to the case when the action of
  $\pi_1(A,x_0)$ is nontrivial, and is stated as such in
  \cite{Hatcher2002Algebraic}.  We will not need that additional
  detail.
\end{rmk}

\begin{defn}
  We say that a map $f\cn X \to Y$ is a
  \emph{homology-$n$-equivalence} if $H_q(f)$ is an isomorphism for
  $q < n$ and a surjection for $q = n$.
\end{defn}

\begin{cor}[Hurewicz for maps]\label{cor:Hurewicz-for-maps}
  Let $f\cn X \to Y$ be a map of path-connected spaces.  If $f$ is an
  $n$-equivalence with $n \ge 1$ then $f$ is a
  homology-$n$-equivalence.  When $X$ and $Y$ are path-connected
  $H$-spaces and $f$ is an $H$-map, then the converse also holds.
\end{cor}
\begin{proof}\proofof{cor:Hurewicz-for-maps}
  Consider the comparison of long exact sequences below.
  The condition that $f$ is an $n$-equivalence is equivalent to the
  condition that $\pi_i(Mf,X,x_0) = 0$ for $i \leq n$.  The condition
  that $f$ is a homology $n$-equivalence is equivalent to the
  condition that $H_i(Mf,X) = 0$ for $i \leq n$.
  \[
  \begin{tikzpicture}[x=30mm,y=15mm]
    \draw[tikzob,mm] 
    (0,1) node (z1){\cdots}
    ++(.7,0) node (a1) {\pi_{i+1}(Mf,X,x_0)}
    ++(1,0) node (b1) {\pi_{i}(X,x_0)}
    ++(1,0) node (c1) {\pi_{i}(Mf,x_0)}
    ++(1,0) node (d1) {\pi_{i}(Mf,X,x_0)}
    ++(.7,0) node (w1) {\cdots}
    (0,0) node (z2) {\cdots}
    ++(.7,0) node (a2) {H_{i+1}(Mf,X)}
    ++(1,0) node (b2) {H_{i}(X)}
    ++(1,0) node (c2) {H_{i}(Mf)}
    ++(1,0) node (d2) {H_{i}(Mf,X)}
    ++(.7,0) node (w2) {\cdots}
    ;
    \path[tikzar,mm] 
    (a1) edge node {} (a2)
    (b1) edge node {} (b2)
    (c1) edge node {} (c2)
    (d1) edge node {} (d2)
    (z1) edge node {} (a1)
    (a1) edge node {} (b1)
    (b1) edge node {} (c1)
    (c1) edge node {} (d1)
    (d1) edge node {} (w1)
    (z2) edge node {} (a2)
    (a2) edge node {} (b2)
    (b2) edge node {} (c2)
    (c2) edge node {} (d2)
    (d2) edge node {} (w2)
    ;
  \end{tikzpicture}
  \]
  Therefore the first statement is a direct consequence of the
  relative Hurewicz theorem.  The second holds also by the relative
  Hurewicz theorem because the assumption that $f$ is an $H$-map
  implies that the induced action of $\pi_1(X,x_0)$ on
  $\pi_1(Mf,X,x_0)$ is trivial \cite[Section~9.2]{Str2011Modern}.
\end{proof}

\begin{proof}[Proof of \cref{prop:n-equiv}]\proofof{prop:n-equiv}
  If $f$ is an $n$-equivalence, each path component of $f$ is an
  $n$-equivalence, so by \cref{cor:Hurewicz-for-maps} each path
  component of $f$ is a homology-$n$-equivalence.  Therefore $f$
  itself is a homology-$n$-equivalence.
  This implies the group-completion, $\Om Bf$, is a
  homology-$n$-equivalence because
  on homology, the map
  \[
    H_* \Om Bf\cn H_*(X)[\pi_0(X)^{-1}] \longrightarrow H_*(Y)[\pi_0(Y)^{-1}]
  \]
  is the map induced by localizing $H_* f$ with respect to $\pi_0$,
  and localization is exact.
  Since $f$ is an $E_\infty$ map, it induces an $H$-map between the
  unit components of $\Om BX$ and $\Om BY$. By the converse part of
  \cref{cor:Hurewicz-for-maps} for the unit component of $\Om B f$,
  this unit component must be an $n$-equivalence.
  Lastly, in any
  group-complete $H$-space, translation by any point $x$ induces a
  homotopy equivalence between the basepoint component and the
  component of $x$.  Hence all of the components of are all homotopy
  equivalent, and therefore $\Om Bf$ is an $n$-equivalence.
\end{proof}

\begin{notn}
  We let $\Sp_{0}$ denote the category of connective spectra,
  i.e., the full subcategory of spectra consisting of those objects
  $X$ with $\pi_n X =0$ for all $n < 0$.
\end{notn}

\begin{defn}
  We say that a map $f\cn X \to Y$ of connective spectra is a
  \emph{stable $P_n$-equivalence} when the conditions of
  \cref{defn:P_n-equiv} hold for stable homotopy groups; i.e., $f$
  induces an isomorphism on stable homotopy groups $\pi_i$ for $0 \leq
  i \leq n$. We let $\mathit{st~P_n\mh eq}$ denote the class of stable
  $P_n$-equivalences.
\end{defn}

\section{Symmetric monoidal algebra in dimension 2}
\label{sec:algebra}

One has a number of distinct notions of symmetric monoidal algebra in
dimension 2, and it will be necessary for us to work with several of
these.  The most general form is the notion of symmetric monoidal
bicategory, and we outline essential details of this structure in
\cref{sec:SM-bicats}.  Several of our constructions make use of a
stricter notion arising as monoids in $\iicat$, and these are reviewed
in \cref{sec:pgm}.

One also has various levels of strength for morphisms, both with
respect to functoriality and with respect to the monoidal structure.
In this paper, we can work solely with those morphisms of symmetric
monoidal bicategories---either strict functors or
pseudofunctors---which preserve the symmetric monoidal structure
strictly (see \cref{defn:strict-monoidal}). In contrast with the
weakest notion of morphism, that of symmetric monoidal pseudofunctor,
these stricter variants all enjoy composition which is strictly
associative and unital.

There are many good reasons to consider versions which are stricter
than the most general possible notion. The most obvious is that the
stricter structures are easier to work with, and in this case often
allow the use of techniques from the highly-developed theory of
2-categories.  The second reason we work with a variety of stricter
notions is that many of these have equivalent \emph{homotopy theories}
to that of the fully weak version; we address this point in
\cref{sec:hty-thy-smb}. Even if some construction does not preserve a
particular strict variant of symmetric monoidal bicategory, but
outputs a different variant with the same homotopy theory, we can
still make use of the stricter setting. Finally, stricter notions
usually admit more transparent constructions; the various $K$-theory
functors for symmetric monoidal bicategories in
\cite{Oso10Spectra,GO2012Infinite,GJO2017KTheory} provide an excellent
example, with stricter variants admitting simpler $K$-theory functors.

\subsection{Background about symmetric monoidal bicategories}
\label{sec:SM-bicats}

In this section we review the minimal necessary content from the
theory of symmetric monoidal bicategories so that the reader can
understand our construction of symmetric monoidal structure from
operad actions in \cref{sec:operads}.  More complete details can be
found in
\cite{McCru00Balanced,SP2011Classification,Lac10Icons,CG2014Iterated}.

\begin{convention}
  We always use \emph{transformation} to mean pseudonatural transformation
  (which we will only indicate via components) and \emph{equivalence} to mean
  pseudonatural (adjoint) equivalence, that is, a pseudonatural
  transformation with inverse up to isomorphisms satisfying triangle identities.
\end{convention}

\begin{defn}[Sketch, see \cite{McCru00Balanced}, 
{\cite[Definition 2.3]{SP2011Classification}}, 
or \cite{CG2014Iterated}]
A \textit{symmetric monoidal bicategory} consists of
\begin{itemize}
\item a bicategory $\cB$,
\item a tensor product pseudofunctor $\cB \times \cB \to \cB$, denoted by
concatenation,
\item a unit object $e \in \ob \cB$,
\item an associativity equivalence $\al\colon (xy)z \simeq x(yz)$,
\item unit equivalences $l\colon ex \simeq x$ and $r\colon x \simeq
xe$,
\item invertible modifications $\pi$, $\mu$, $\la$, $\rho$ as follows, 
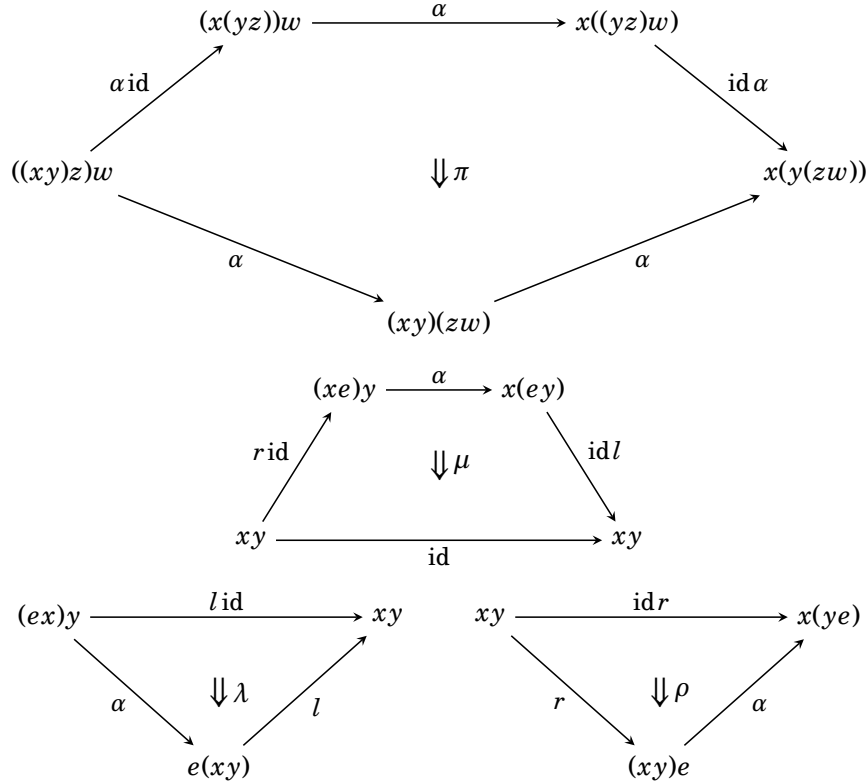
\[
  \begin{tikzpicture}[x=25mm,y=20mm]
    \draw[tikzob,mm] 
    (0,0) node (0) {((xy)z)w}
    (1,1) node (1) {(x(yz))w}
    (3,1) node (2) {x((yz)w)}
    (4,0) node (3) {x(y(zw))}
    (2,-1) node (4) {(xy)(zw)};
    \path[tikzar,mm] 
    (0) edge node {\al \,\id} (1)
    (1) edge node {\al} (2)
    (2) edge node {\id\,\al} (3)
    (0) edge[swap] node {\al} (4)
    (4) edge[swap] node {\al} (3);
    \draw[tikzob,mm]
    (2,0) node[rotate=270,font=\Large] {\Rightarrow}
    ++(3mm,0mm) node {\pi};
  \end{tikzpicture}
  \]
    \[
  \begin{tikzpicture}[x=25mm,y=20mm]
    \draw[tikzob,mm] 
    (0,0) node (0) {(xe)y}
    (1,0) node (1) {x(ey)}
    (-.5,-1) node (2) {xy}
    (1.5,-1) node (3) {xy};
    \path[tikzar,mm] 
    (0) edge node {\al} (1)
    (2) edge node {r\,\id} (0)
    (2) edge[swap] node {\id} (3)
    (1) edge node {\id \, l} (3);
    \draw[tikzob,mm]
    (0.5,-0.5) node[rotate=270,font=\Large] {\Rightarrow}
    ++(3mm,0mm) node {\mu};
  \end{tikzpicture}
  \]
    \[
  \begin{tikzpicture}[x=45mm,y=20mm]
    \draw[tikzob,mm] 
    (0,0) node (0) {(ex)y}
    (1,0) node (1) {xy}
    (.5,-1) node (2) {e(xy)};
    \path[tikzar,mm] 
    (0) edge node {l\,\id} (1)
    (0) edge[swap] node {\al} (2)
    (2) edge[swap] node {l} (1);
    \draw[tikzob,mm]
    (0.5,-0.5) node[rotate=270,font=\Large] {\Rightarrow}
    ++(3mm,0mm) node {\la};
  \end{tikzpicture}
  \qquad
    \begin{tikzpicture}[x=45mm,y=20mm]
    \draw[tikzob,mm] 
    (0,0) node (0) {xy}
    (1,0) node (1) {x(ye)}
    (.5,-1) node (2) {(xy)e};
    \path[tikzar,mm] 
    (0) edge node {\id\, r} (1)
    (0) edge[swap] node {r} (2)
    (2) edge[swap] node {\al} (1);
    \draw[tikzob,mm]
    (0.5,-0.5) node[rotate=270,font=\Large] {\Rightarrow}
    ++(3mm,0mm) node {\rho};
  \end{tikzpicture}
  \]
\item a braid equivalence $\beta\colon xy \simeq yx$,
\item two invertible modifications (denoted $R_{-|--},$ $R_{--|-}$)
  which correspond to two instances of the third Reidemeister move,
    
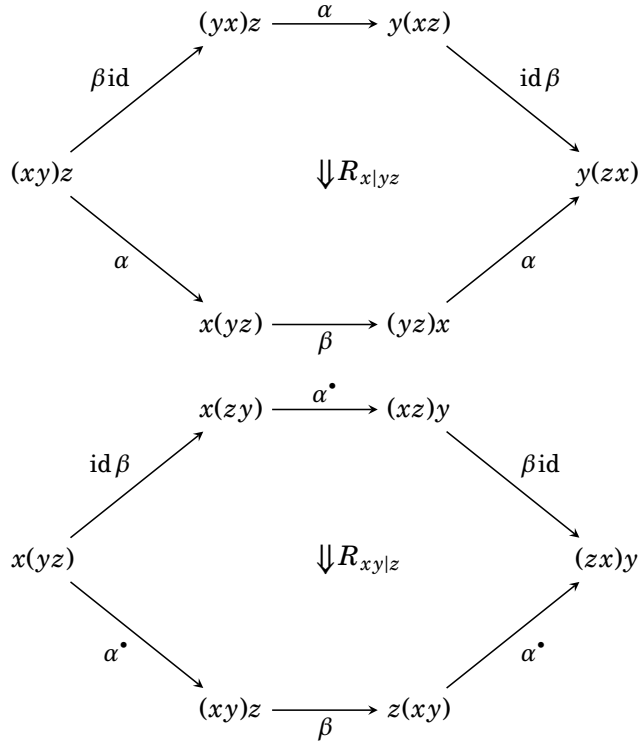
\[
  \begin{tikzpicture}[x=25mm,y=20mm]
    \draw[tikzob,mm] 
    (0,0) node (0) {(xy)z}
    (1,1) node (1) {(yx)z}
    (2,1) node (2) {y(xz)}
    (3,0) node (3) {y(zx)}
    (1,-1) node (4) {x(yz)}
    (2,-1) node (5) {(yz)x};
    \path[tikzar,mm] 
    (0) edge node {\be \,\id} (1)
    (1) edge node {\al} (2)
    (2) edge node {\id\,\be} (3)
    (0) edge[swap] node {\al} (4)
    (4) edge[swap] node {\be} (5)
    (5) edge[swap] node {\al} (3);
    \draw[tikzob,mm]
    (1.5,0) node[rotate=270,font=\Large] {\Rightarrow}
    ++(6mm,0mm) node {R_{x \mid yz}};
  \end{tikzpicture}
  \]
      \[
  \begin{tikzpicture}[x=25mm,y=20mm]
    \draw[tikzob,mm] 
    (0,0) node (0) {x(yz)}
    (1,1) node (1) {x(zy)}
    (2,1) node (2) {(xz)y}
    (3,0) node (3) {(zx)y}
    (1,-1) node (4) {(xy)z}
    (2,-1) node (5) {z(xy)};
    \path[tikzar,mm] 
    (0) edge node {\id\,\be} (1)
    (1) edge node {\al^\bullet} (2)
    (2) edge node {\be\,\id} (3)
    (0) edge[swap] node {\al^\bullet} (4)
    (4) edge[swap] node {\be} (5)
    (5) edge[swap] node {\al^\bullet} (3);
    \draw[tikzob,mm]
    (1.5,0) node[rotate=270,font=\Large] {\Rightarrow}
    ++(6mm,0mm) node {R_{xy\mid z}};
  \end{tikzpicture}
  \]
\item and an invertible modification (the syllepsis, $v$)
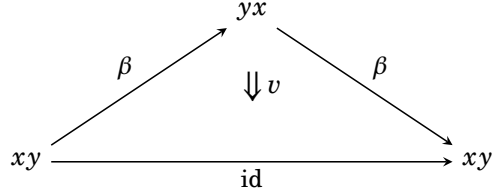
    \[
  \begin{tikzpicture}[x=60mm,y=20mm]
    \draw[tikzob,mm] 
    (0,-1) node (0) {xy}
    (1,-1) node (1) {xy}
    (.5,0) node (2) {yx};
    \path[tikzar,mm] 
    (0) edge[swap] node {\id} (1)
    (0) edge node {\be} (2)
    (2) edge node {\be} (1);
    \draw[tikzob,mm]
    (0.5,-0.5) node[rotate=270,font=\Large] {\Rightarrow}
    ++(3mm,0mm) node {v};
  \end{tikzpicture}
  \]
\end{itemize}
satisfying three axioms for the monoidal structure, four axioms for
the braided structure, two axioms for the sylleptic structure, and one
final axiom for the symmetric structure.
\end{defn}

\begin{defn}[Sketch, see {\cite[Definition 2.5]{SP2011Classification}}]\label{defn:symm-mon-psfun}
A \textit{symmetric monoidal pseudofunctor} $F\colon\cB \to \cC$
consists
  of
  \begin{itemize}
  \item a pseudofunctor $F\colon \cB \to \cC$,
  \item a unit equivalence $e_{\cC} \simeq F(e_{\cB})$,
  \item an equivalence for the tensor product $Fx Fy \simeq F(xy)$,
  \item three invertible modifications between composites of the unit
    and tensor product equivalences, and
  \item an invertible modification comparing the braidings in
    $\cB$ and $\cC$
  \end{itemize}
  satisfying two axioms for the monoidal structure, two axioms for the
  braided structure, and one axiom for the symmetric (and hence
  subsuming the sylleptic) structure.
\end{defn}

\begin{defn}[Sketch, see {\cite[Definition 2.7]{SP2011Classification}}]
  A \textit{symmetric monoidal transformation} $\eta\colon F \rtarr G$
  consists of
  \begin{itemize}
  \item a transformation $\eta \colon F \rtarr G$, and
  \item two invertible modifications concerning the interaction
    between $\eta$ and the unit objects on the one hand and the tensor
    products on the other
  \end{itemize}
  satisfying two axioms for the monoidal structure and one axiom for
  the symmetric structure (and hence subsuming the braided and
  sylleptic structures).
\end{defn}

The following is verified in \cite{SP2011Classification}.  Note that we
have not defined symmetric monoidal modifications as we will not have
any reason to use them in any of our constructions.

\begin{lem}
There is a tricategory $\mathsf{SMB}$ of symmetric monoidal
bicategories, symmetric
  monoidal pseudofunctors, symmetric monoidal transformations, and
  symmetric monoidal modifications.
\end{lem}

We will need to know when symmetric monoidal pseudofunctors or
transformations are invertible in the appropriate sense.

\begin{defn}
  A \textit{symmetric monoidal biequivalence} $F\colon \cB \to \cC$ is
  a symmetric monoidal pseudofunctor such that the underlying
  pseudofunctor $F$ is a biequivalence of bicategories, i.e., has an
  inverse up to pseudonatural equivalence.
\end{defn}

\begin{defn}
  \label{defn:symm-mon-equiv}
A \textit{symmetric monoidal equivalence} $\eta \colon F \rtarr G$
between
  symmetric monoidal pseudofunctors is a symmetric monoidal
  transformation $\eta \colon F \rtarr G$ such that the underlying
  transformation $\eta$ is an equivalence.  This is logically
  equivalent to the condition that each component 1-cell
  $\eta_{b}\colon Fb \to Gb$ is an equivalence 1-cell in $\cC$.
\end{defn}

The results of \cite{Gur2012Biequivalences} can be used to easily prove
the following lemma, although the first part is also verified by
elementary means in \cite{SP2011Classification}.

\begin{lem}
  Let $F,G\colon \cB \to \cC$ be symmetric monoidal pseudofunctors,
  and $\eta\colon F \rtarr G$ a symmetric monoidal transformation
  between them.
  \begin{itemize}
  \item $F\colon \cB \to \cC$ is a symmetric monoidal biequivalence if
    and only if it is an internal biequivalence in the tricategory
    $\mathsf{SMB}$.
  \item $\eta\colon F \rtarr G$ is a symmetric monoidal equivalence if
    and only if it is an internal equivalence in the bicategory
    $\mathsf{SMB}(\cB,\cC)$.
  \end{itemize}
\end{lem}

We have defined a symmetric monoidal biequivalence to be a symmetric
monoidal pseudofunctor which is also a biequivalence. The content of this
lemma is that the weak inverse can also be chosen to be symmetric
monoidal, as well as all the accompanying transformations and
modifications.

\begin{defn}\label{defn:HoSMB}
  Let $\Ho \SMB$ denote the category of symmetric monoidal
  bicategories with morphisms given by equivalence classes of
  symmetric monoidal pseudofunctors under the relation given by
  symmetric monoidal pseudonatural equivalence. Note that in this
  category, every symmetric monoidal biequivalence is an isomorphism.
\end{defn}

\begin{defn}\label{defn:strict-monoidal}
  A \textit{strictly symmetric monoidal pseudofunctor}
  $F\colon \cB \to \cC$ between symmetric monoidal bicategories is a
  pseudofunctor of the underlying bicategories that preserves the
  symmetric monoidal structure strictly, and for which all of the
  constraints are either the identity (when this makes sense) or the
  unique coherence isomorphism obtained from the coherence theorem for
  pseudofunctors \cite{JS1993btc,Gurski13Coherence}.  A \emph{strict
    functor} is a strictly symmetric monoidal pseudofunctor for which
  the underlying pseudofunctor is strict.
\end{defn}

\begin{rmk}
  There is a monad on the category of 2-globular sets whose algebras
  are symmetric monoidal bicategories.  Strict functors can then be
  identified with the morphisms in the Eilenberg-Moore category for
  this monad, and in particular symmetric monoidal bicategories with
  strict functors form a category.  This point of view is crucial to
  the methods employed in \cite{SP2011Classification}.
\end{rmk}

\begin{notn}
  The principal variants we will use are listed below.
  \begin{itemize}
  \item We let $\SMBicat_{ps}$ denote the category of symmetric
    monoidal bicategories and strictly symmetric monoidal
    pseudofunctors.  Note that the composition of these is given by
    the composite of the underlying pseudofunctors and then the unique
    choice of coherence cells making them strictly symmetric.
  \item We let $\SMBicat_{s}$ denote the subcategory of
    $\SMBicat_{ps}$ whose morphisms are strict functors.
  \item We let $\SMIICat_{ps}$, respectively $\SMIICat_{s}$,
    denote the full subcategories of $\SMBicat_{ps}$, respectively
    $\SMBicat_{s}$, with objects whose underlying bicategory is a
    2-category.
\end{itemize}
\end{notn}

We note two subtleties regarding subcategories of strict functors.
The first is that the inverse of a strictly symmetric monoidal strict
biequivalence is not necessarily itself strict.  However, we will see
in \cref{cor:smbs=smbps} that the homotopy category obtained by
inverting strict biequivalences in $\SMBicat_{s}$ is equivalent to
$\Ho \SMB$.

Second, note that the multiplication map of a symmetric monoidal
bicategory or 2-category $\cA$ is a pseudofunctor
\[
  \cA \times \cA \to \cA.
\]
In both $\SMBicat_{s}$ and $\SMIICat_{s}$, we consider strictly
functorial morphisms which commute strictly with this multiplication
pseudofunctor. The work in \cite{GJO2017KTheory} shows that, relative
to all stable equivalences, it is possible to restrict the structure
further and still represent every stable homotopy type.  Relative only
to the categorical equivalences, however, we must retain some
pseudofunctoriality in the multiplication.

\subsection{Background on permutative Gray-monoids}
\label{sec:pgm}

In this section we give a definition that is a semi-strict version of
symmetric monoidal bicategories. Here too we give the minimal
necessary background for our current work.  For details, see
\cite{Gra74Formal,GPS95Coherence,Gurski13Coherence}, or \cite[Section 3]{GJO2017KTheory}.

\begin{defn}\label{defn:graytensor}
  Let $\cA, \cB$ be 2-categories.  The \emph{Gray tensor product} of $\cA$
  and $\cB$, written $\cA \otimes \cB$ is the 2-category given by
  \begin{itemize}
  \item 0-cells consisting of pairs $a \otimes b$ with $a$ an object
    of $\cA$ and $b$ an object of $\cB$;
  \item 1-cells generated under composition by basic 1-cells of the
    form $f \otimes 1: a \otimes b \to a' \otimes b$ for $f:a \to a'$
    in $\cA$ and $1 \otimes g: a \otimes b \to a \otimes b'$ for
    $g:b \to b'$ in $\cB$; and
  \item 2-cells generated by basic 2-cells of the form
    $\al \otimes 1$ for 2-cells $\al$ in $\cA$; $1 \otimes \de$ for
    2-cells $\de$ in $\cB$; and new 2-cells
    $\Si_{f,g}: (f \otimes 1)(1 \otimes g) \cong (1 \otimes g)(f
    \otimes 1)$.
  \end{itemize}
  These cells satisfy axioms related to composition, naturality and
  bilinearity; for a complete list, see
  \cite[Section 3.1]{Gurski13Coherence} or
  \cite[Definition 3.16]{GJO2017KTheory}.
\end{defn}

The assignment $(\cA,\cB)\mapsto \cA \otimes \cB$ extends to a functor
of categories
\[
\IICat \times \IICat \rtarr \IICat
\]
which defines a symmetric monoidal structure on $\IICat$. The unit for
this monoidal structure is the terminal 2-category.
The Gray tensor product has a universal property that relates it to
the notion of cubical functor.

\begin{defn}\label{defn:cubical}
  Let $\cA_1$, $\cA_2$ and $\cB$ be 2-categories. A \emph{cubical
    functor} $F\colon \cA_1 \times \cA_2 \rtarr \cB$ is a normal
  pseudofunctor such that for all composable pairs $(f_1,f_2)$,
  $(g_1,g_2)$ of 1-cells in $\cA_1\times\cA_2$, the comparison 2-cell
  \[ 
    \phi \colon F(f_1,f_2)\circ F(g_1,g_2) \Rightarrow F(f_1\circ g_1,f_2\circ g_2)
  \]
  is the identity whenever either $f_1$ or $g_2$ is the identity.
\end{defn}

\begin{thm}[{\cite[Theorem 3.7]{Gurski13Coherence}, \cite[Theorem 3.21]{GJO2017KTheory}}]
  \label{thm:cubical-Gray}
  Let $\cA$, $\cB$ and $\cC$ be 2-categories. There is a cubical
  functor \[ c\colon \cA \times \cB \rtarr \cA\otimes\cB
  \]
  natural in $\cA$ and $\cB$, such that composition with $c$ induces a
  bijection between cubical functors $\cA\times \cB \rtarr \cC$ and
  2-functors $\cA\otimes \cB \rtarr \cC$.
\end{thm}

\begin{rmk}\label{rem:id-is-lax-monoidal}
  There exists a 2-functor
  $i\cn \cA \otimes \cB \rtarr \cA \times \cB$ natural in $\cA$ and
  $\cB$ such that $i\circ c=\id$ and $c\circ i \cong \id$ (see
  \cite[Corollary 3.22]{Gurski13Coherence}). This map makes the
  identity functor $\Id$ on $\iicat$ a lax symmetric monoidal functor
  \[
    (\iicat, \times) \rtarr (\iicat,\otimes)
  \]
  with the constraint
  \[
  \Id(\cA) \otimes \Id(\cB) \to \Id(\cA \times \cB)
  \]
  given by $i$.  Similarly, $c$ gives the constraint that makes the
  inclusion
  \[
    (\iicat,\otimes) \rtarr (\iicatps,\times)
  \]
  into a lax symmetric monoidal functor \cite{Gur2013monoidal}.
\end{rmk}

\begin{defn}
  \label{defn:gray-monoid}
  A \emph{Gray-monoid} is a monoid object in $(\IICat,\otimes)$.  This
  consists of a 2-category $\cC$, a 2-functor
  \[
  \oplus\cn \cC \otimes \cC \to \cC,
  \]
  and an object $e$ of $\cC$ satisfying associativity and unit
  axioms.
\end{defn}

Via the bijection in \cref{thm:cubical-Gray}, we can view a
Gray-monoid as a particular type of monoidal bicategory such that the
monoidal product is a cubical functor and all the other coherence 
cells are identities \cite[Theorem 8.12]{Gurski13Coherence}.

\begin{defn}\label{defn:pgm}
  A \textit{permutative Gray-monoid} $\cC$ consists of a Gray-monoid
  $(\cC, \oplus, e)$ together with a 2-natural isomorphism,
  \[
    \begin{tikzpicture}[x=1mm,y=1mm]
    \draw[tikzob,mm] 
    (0,0) node (00) {\cC \otimes \cC}
    (25,0) node (10) {\cC \otimes \cC}
    (12.5,-10) node (01) {\cC}
    ;
    \path[tikzar,mm] 
    (00) edge node {\tau} (10)
    (10) edge node {\oplus} (01)
    (00) edge[swap] node {\oplus} (01)
    ;
    \draw[tikzob,mm]
    (12.5,-4) node {\Anglearrow{40} \beta}
    ;
  \end{tikzpicture}
  \]
  where $\tau \cn \cC \otimes \cC \to \cC \otimes \cC$ is the symmetry
  isomorphism in $\IICat$ for the Gray tensor product, such that the
  following axioms hold.
  \begin{itemize}
  \item The following pasting diagram is equal to the identity
    2-natural transformation for the 2-functor $\oplus$.
    \[
    \begin{tikzpicture}[x=1mm,y=1mm]
    \draw[tikzob,mm] 
    (0,0) node (00) {\cC \otimes \cC}
    (25,0) node (10) {\cC \otimes \cC}
    (50,0) node (20) {\cC \otimes \cC}
    (25,-15) node (11) {\cC}
    ;
    \path[tikzar,mm] 
    (00) edge node{\tau} (10)
    (10) edge node{\tau} (20)
    (00) edge[swap] node{\oplus} (11)
    (10) edge[swap] node{\oplus} (11)
    (20) edge node{\oplus} (11)
    (00) edge[bend left] node{\id} (20)
    ;
    \draw[tikzob,mm] 
    (14.5,-4) node {\scriptstyle \Anglearrow{40} \beta}
    (35.5,-4) node {\scriptstyle \Rightarrow \beta}    
    ;
    \end{tikzpicture}
    \]

  \item The following equality of pasting diagrams holds where we have
    abbreviated the tensor product to concatenation when labeling 1-
    or 2-cells.
    \[
    \begin{tikzpicture}[x=.95mm,y=1mm]
    \draw[tikzob,mm] 
    (3,-10) node (00) {\cC^{\otimes 3}}
    (18,0) node (10) {\cC^{\otimes 3}}
    (36,0) node (20) {\cC^{\otimes 3}}
    (51,-10) node (30) {\cC^{\otimes 2}}
    (27,-15) node (11) {\cC^{\otimes 2}}
    (18,-30) node (12) {\cC^{\otimes 2}}
    (36,-30) node (33) {\cC}
    (69,-10) node (40) {\cC^{\otimes 3}}
    (84,0) node (50) {\cC^{\otimes 3}}
    (102,0) node (60) {\cC^{\otimes 3}}
    (117,-10) node (70) {\cC^{\otimes 2}}
    (84,-30) node (52) {\cC^{\otimes 2}}
    (102,-30) node (73) {\cC}
    (102,-17) node (63) {\cC^{\otimes 2}}
    ;
    \path[tikzar,mm] 
    (00) edge node{\scriptstyle \tau \id} (10)
    (40) edge node{\scriptstyle \tau \id} (50)
    (10) edge node{\scriptstyle \id \tau } (20)
    (50) edge node{\scriptstyle \id \tau } (60)
    (20) edge node{\scriptstyle \oplus \id} (30)
    (60) edge node{\scriptstyle \oplus \id} (70)
    (30) edge node{\scriptstyle \oplus} (33)
    (70) edge node{\scriptstyle \oplus} (73)
    (00) edge[swap] node{\scriptstyle \oplus \id} (12)
    (40) edge[swap] node{\scriptstyle \oplus \id} (52)
    (12) edge[swap] node{\scriptstyle \oplus} (33)
    (52) edge[swap] node{\scriptstyle \oplus} (73)
    (00) edge[swap] node{\scriptstyle \id \oplus} (11)
    (11) edge node{\scriptstyle \tau} (30)
    (11) edge[swap] node{\scriptstyle \oplus} (33)
    (50) edge node{\scriptstyle \oplus \id} (52)
    (50) edge[swap] node{\scriptstyle \id \oplus} (63)
    (60) edge node{\scriptstyle \id \oplus} (63)
    (63) edge node{\scriptstyle \oplus} (73)
    ;
    \draw[tikzob,mm] 
    (27,-7.5) node {=}
    (19,-21) node {=}
    (108,-12) node {=}
    (93,-20) node {=}
    (59,-15) node {=}
    (37,-19) node {\scriptstyle \Anglearrow{40} \beta}
    (78.4,-12.5) node {\scriptstyle \Anglearrow{40} \beta \id}
    (96,-5) node {\scriptstyle \Anglearrow{40} \id \beta}
    ;
    \end{tikzpicture}
    \]
  \end{itemize}
\end{defn}

\begin{rmk}
  In \cite{GJO2017KTheory,GJOS2017Postnikov} the definition of
  permutative Gray-monoid includes a third axiom relating $\be$ to the
  unit $e$. This axiom is implied by the other two axioms and is
  therefore unnecessary.
\end{rmk}

\begin{defn}
  \label{defn:strict-functor-gray-mon}
  A \textit{strict functor} $F:\cC \to \cD$ of permutative
  Gray-monoids is a 2-functor $F:\cC \to \cD$ of the underlying
  2-categories satisfying the following conditions.
  \begin{itemize}
  \item $F(e_\cC) = e_\cD$, so that $F$ strictly preserves the unit
    object.
  \item The diagram
    \[
    \begin{tikzpicture}[x=1mm,y=1mm]
    \draw[tikzob,mm] 
    (0,0) node (00) {\cC \otimes \cC}
    (30,0) node (10) {\cD \otimes \cD}
    (0,-15) node (01) {\cC}
    (30,-15) node (11) {\cD}
    ;
    \path[tikzar,mm] 
    (00) edge node{F \otimes F} (10)
    (10) edge node{\oplus_{\cD}} (11)
    (00) edge[swap] node{\oplus_{\cC}} (01)
    (01) edge[swap] node{F} (11)    
    ;
    \end{tikzpicture}
    \]
    commutes, so that $F$ strictly preserves the sum.
  \item The equation
    \[
    \beta^{\cD} * (F \otimes F) = F * \beta^{\cC}
    \]
    holds, so that $F$ strictly preserves the symmetry.  This equation
    is equivalent to requiring that
    \[
    \beta^{\cD}_{Fx,Fy} = F(\beta^{\cC}_{x,y})
    \]
    as 1-cells from $Fx \oplus Fy = F(x \oplus y)$ to
    $Fy \oplus Fx = F(y \oplus x)$.
  \end{itemize}
\end{defn}

\begin{notn}\label{notn:PGM}
  The category of permutative Gray-monoids, $\PGM$, is the full
  subcategory of $\SMIICat_{s}$ whose objects are permutative
  Gray-monoids.
\end{notn}

The following two results follow from straightforward calculations and
are used in \cref{sec:choice-of-mult}.
\begin{prop}
  \label{prop:defn-PGM}
  The underlying 2-category functor $\PGM \to \IICat$ is monadic in
  the usual, 1-categorical sense.
\end{prop}

\begin{lem}\label{lem:psfun-out-of-product}
  Let $F \cn X \times Y \to Z$ be a pseudofunctor between
  bicategories.
  \begin{itemize}
  \item For any object $x$ of $X$, $F$ induces a pseudofunctor
    $F(x,-) \cn Y \to Z$. The pseudofunctor $F(x,-)$ is strict if $F$
    is, hence a 2-functor if $X,Y$ are 2-categories.
  \item For any 1-cell $f \cn x \to x'$, $F$ induces a pseudonatural
    transformation $F(f,-)$ from $F(x,-)$ to $F(x',-)$; if $f$ is an
    equivalence in $X$, then the pseudonatural transformation $F(f,-)$
    is an equivalence. The transformation $F(f,-)$ is strict if $F$
    is, hence a 2-natural transformation if $F$ is strict and $X,Y$
    are 2-categories; furthermore, if $f$ is also an isomorphism then
    $F(f,-)$ is a 2-natural isomorphism.
  \item For any 2-cell $\al \cn f \Rightarrow f'$, $F$ induces a
    modification $F(\al, -)$ from $F(f,-)$ to $F(f',-)$; if $\al$ is
    invertible in $X$, then the modification $F(\al, -)$ is an
    isomorphism.
  \end{itemize}
\end{lem}

One uses the modifier ``Picard'' for symmetric monoidal algebra where
all objects and morphisms are invertible.  We have several notions in
dimension 2, each consisting of those objects which have invertible
0-, 1-, and 2-cells.

\begin{defn}\label{defn:invertible2}
  Let $(\cD, \oplus, e)$ be a Gray-monoid.
  \begin{enumerate}
  \item A 2-cell of $\cD$ is invertible if it has an inverse in the
    usual sense.
  \item A 1-cell $f \cn x \to y$ is invertible if there exists a
    1-cell $g \cn y \to x$ together with invertible 2-cells
$g\circ f \cong \id_{x}$, $f\circ g \cong \id_{y}$. In other words,
$f$ is
    invertible if it is an internal equivalence (denoted with the
    $\simeq$ symbol) in $\cD$.
  \item An object $x$ of $\cD$ is invertible if there exists another
object $y$ together with invertible 1-cells $x \oplus y \simeq e$,    $y \oplus x \simeq e$.
  \end{enumerate}
\end{defn}

\begin{notn}[Picard objects in dimension 2 {\cite[Definition 2.19]{GJOS2017Postnikov}}]
  \ 
  \begin{itemize}
  \item $\Pic\Bicat_{s}$ denotes the full subcategory of
    $\SMBicat_{s}$ consisting of those symmetric monoidal bicategories
    with all cells invertible; we call these \emph{Picard bicategories}.
  \item $\Pic\IICat_{s}$ denotes the full subcategory of
$\SMIICat_{s}$ consisting of symmetric monoidal 2-categories with    
all cells invertible; we call these \emph{Picard 2-categories}.
\item $\PicPGM$ denotes the full subcategory of $\PGM$ consisting of   
 those permutative Gray-monoids with all cells 
    invertible; we call these \emph{strict Picard
      2-categories}.
  \end{itemize}
\end{notn}

\subsection{The homotopy theory of symmetric monoidal bicategories}
\label{sec:hty-thy-smb}

In this section we discuss the homotopy theories for symmetric
monoidal algebra in dimension 2 and obtain a number of equivalence
results. To begin, we recall quasistrictification results from
\cite{SP2011Classification} and \cite{GJO2017KTheory}, which show how
to replace a symmetric monoidal bicategory with an appropriately
equivalent permutative Gray-monoid.
\begin{thm}[{\cite[Theorem 2.97]{SP2011Classification}, \cite[Theorem 3.14]{GJO2017KTheory}}]
  \label{cohqs2cats}
  Let $\cB$ be a symmetric monoidal bicategory.
  \begin{enumerate}
  \item There are two endofunctors, $\cB \mapsto \cB^{c}$ and
    $\cB \mapsto \cB^{qst}$, of $\SMBicat_{s}$.  Any symmetric
    monoidal bicategory of the form $\cB^{qst}$ is a 
    permutative Gray-monoid.
  \item There are natural transformations
    $(-)^{c} \impl \id, (-)^{c} \impl (-)^{qst}$.  When evaluated at a
    symmetric monoidal bicategory $\cB$, these give natural strict
    biequivalences
    \[
    \cB \leftarrow \cB^{c} \to \cB^{qst}.
    \]
  \item For a symmetric monoidal pseudofunctor $F:\cB \to \cC$, there
    are strict functors $F^{c}:\cB^{c} \to \cC^{c}$,
    $F^{qst}:\cB^{qst} \to \cC^{qst}$ such that the right hand square
    below commutes and the left hand square commutes up to a symmetric
    monoidal equivalence.
    \[
    \begin{tikzpicture}[x=1mm,y=1mm]
    \draw[tikzob,mm] 
    (0,0) node (00) {\cB}
    (30,0) node (10) {\cB^{c}}
    (60,0) node (20) {\cB^{qst}}
    (0,-12) node (01) {\cC}
    (30,-12) node (11) {\cC^{c}}
    (60,-12) node (21) {\cC^{qst}}
    ;
    \path[tikzar,mm] 
    (10) edge node{F^{c}} (11)
    (10) edge node{ } (20)
    (20) edge node{F^{qst}} (21)
    (10) edge node{ } (00)
    (00) edge[swap] node{F} (01)
    (11) edge node{ } (01)
    (11) edge node{ } (21)
    ;
    \draw[tikzob,mm] 
    (15,-6) node {\simeq}
    ;
    \end{tikzpicture}
    \]
  \end{enumerate}
\end{thm}

\begin{thm}[{\cite[Theorem 3.15]{GJO2017KTheory}}]
  \label{thm:cohqs2cats2}
  Let $\cB$ be a  
    permutative Gray-monoid.
  \begin{enumerate}
  \item There is a strict functor $\nu \cn \cB^{qst} \to \cB$ such
    that
    \[
    \begin{tikzpicture}[x=1mm,y=1mm]
    \draw[tikzob,mm] 
    (0,0) node (0) {\cB^{c}}
    (40,0) node (1) {\cB^{qst}}
    (20,-12) node (2) {\cB}
    ;
    \path[tikzar,mm] 
    (0) edge (1)
    (1) edge node{\nu} (2)
    (0) edge (2)
    ;
    \end{tikzpicture}
    \]
    commutes, where the unlabeled morphisms are those from
    \cref{cohqs2cats}.  In particular, $\nu$ is a strict symmetric
    monoidal biequivalence.
  \item For a symmetric monoidal pseudofunctor $F:\cB \to \cC$ with
    $\cB, \cC$ with both
    permutative Gray-monoids, the square below commutes up to
    a symmetric monoidal equivalence.
    \[
    \begin{tikzpicture}[x=1mm,y=1mm]
    \draw[tikzob,mm] 
    (0,0) node (00) {\cB^{qst}}
    (40,0) node (10) {\cB}
    (0,-12) node (01) {\cC^{qst}}
    (40,-12) node (11) {\cC}
    ;
    \path[tikzar,mm] 
    (00) edge node{\nu} (10)
    (10) edge node{F} (11)
    (00) edge[swap] node{F^{qst}} (01)
    (01) edge[swap] node{\nu} (11)
    ;
    \draw[tikzob,mm] 
    (20,-6) node {\simeq}
    ;
    \end{tikzpicture}
    \]
  \end{enumerate}
\end{thm}

We also require an additional detail about quasistrictification which
follows from the construction in \cite[Theorem
2.97]{SP2011Classification} and the proof of [\textit{loc.~cit.},
Proposition 2.77].
\begin{lem}\label{lem:qst-detail} 
  Let $\cA$, $\cB$ and $\cC$ be permutative Gray-monoids, and let
  $F\cn \cA \to \cB$ and $G\cn \cB \to \cC$ be symmetric monoidal
  pseudofunctors. Then $(GF)^{qst}$ is symmetric monoidal equivalent
  to $G^{qst}F^{qst}$. Moreover, if $G$ is a strict functor, then
  $(GF)^{qst} = G^{qst}F^{qst}$.
\end{lem}

We now turn to equivalences of homotopy theories.  In
\cref{prop:smbs=smbps,cor:smbs=smbps} we show that the homotopy category of $\SMBicat_s$ 
is equivalent to $\Ho\SMB$. In \cref{lem:PGM-to-SMB-equiv} we show
that the homotopy theory of $\Pic\PGM$, respectively $\PGM$, is
equivalent to that of $\Pic\SMBicat_{s}$, respectively $\SMBicat_{s}$,
with a wide choice of classes of weak equivalences.

\begin{defn}
  Let $\mathsf{PGM}$ denote the tricategory whose objects are
  permutative Gray-monoids and whose 1-, 2-, and 3-cells are those of
  $\SMB$.  Likewise let $\Ho\mathsf{PGM}$ denote the full subcategory
  of $\Ho\SMB$ whose objects are permutative Gray-monoids; its
  morphisms are given by equivalence classes of pseudofunctors.
\end{defn}

\begin{defn}
For a category $\cC$ and a class of morphisms $\cW$ in $\cC$, we write 
$\Ho(\cC, \cW)$ for the localization of $\cC$ with respect to $\cW$; we note
that this is not necessarily locally small. This category is the homotopy category
of the homotopy theory $(\cC, \cW)$.
See \cite{Rez01Model} for the general theory supporting this notion,
or \cite[Section 2.1]{GJO2017KTheory} for an overview.
\end{defn}

\begin{prop}\label{prop:smbs=smbps}
  There is an isomorphism of categories
  \[
    \Ho(\PGM, \cateq) \iso \Ho\mathsf{PGM}.
  \]
\end{prop}
\begin{proof}\proofof{prop:smbs=smbps}\marginref{cohqs2cats,thm:cohqs2cats2}
  We define a functor 
  \[
    \Phi\cn\Ho \mathsf{PGM} \to \Ho( \PGM, \cateq) 
  \]
  to be the identity on objects.    For permutative
  Gray-monoids $\cA$ and $\cB$,
  let
  $[\cA, \cB]$ denote $\Ho\mathsf{PGM}(\cA, \cB)$, and let
  $\{\cA,\cB\}$ denote $\Ho (\PGM, \cateq)(\cA, \cB)$.
  As a mnemonic, note that $[\cA,\cB]$ is a set of equivalence classes
  of morphisms, while $\{\cA,\cB\}$ is, \textit{a~priori}, defined by
  zigzags of strict functors.  
  We define a function
  \[
    \Phi_{\cA,\cB} \cn [\cA, \cB] \to \{\cA, \cB\}
  \]
  below, using the quasistrictification $(-)^{qst}$ of
  \cref{cohqs2cats}.
  
  Recall that when $\cA$ is a permutative
  Gray-monoid, there is a strict symmetric monoidal biequivalence
  $\nu \cn \cA^{qst} \to \cA$ which is natural in strict functors
  (\cref{thm:cohqs2cats2}).  For a symmetric monoidal pseudofunctor
  $F \cn \cA \to \cB$, we let $\Phi_{\cA,\cB}(F)$ be the class of the
  zigzag
  \[
    \begin{tikzpicture}[x=14mm,y=20mm]
      \draw[tikzob,mm] 
      (0,0) node (A) {\cA}
      ++(1,0) node (Aq) {\cA^{qst}}
      ++(1.6,0) node (Bq) {\cB^{qst}}
      ++(1,0) node (B) {\cB}
      ;
      \path[tikzar,mm] 
      (Aq) edge[swap] node {\nu} (A)
      (Aq) edge node {F^{qst}} (Bq)
      (Bq) edge node {\nu} (B)
      ;
    \end{tikzpicture}
  \]
  in $\Ho( \PGM, \cateq)$. We must check that this function is
  well-defined on equivalence classes of morphisms which will be
  used to prove that the functions $\Phi_{\cA,\cB}$ assemble to define a functor.
    We will show that this functor is a bijection, hence $\Phi$ defines an
  isomorphism of categories.

  We begin with an observation about the collection of functions
  $\Phi_{-,-}$.
  Let $F\colon \cA \to \cB$ be a symmetric monoidal pseudofunctor and
  $S\colon \cB \to \cC$ is a strict functor.
  By \cref{cohqs2cats}, \cref{lem:qst-detail}, and the
  naturality of $\nu$ with respect to strict functors, we have the
  following equalities in $\Ho(\PGM, \cateq)$:
  \begin{equation}\label{eq:Phi-SF}
    \Phi_{\cA,\cC}(SF) = \nu (SF)^{qst} \nu^{-1} =  \nu S^{qst}F^{qst}
    \nu^{-1} = S \nu F^{qst}    \nu^{-1} = S \Phi_{\cA,\cB}(F).
  \end{equation}

  We now apply the observation in \cref{eq:Phi-SF} to a path object
  construction.  Given a symmetric monoidal bicategory $\cB$, the path
  object $\cB^I$ is another symmetric monoidal bicategory with
  \begin{itemize}
  \item objects $(b,b', f)$ where $b,b'$ are objects and
    $f \cn b \to b'$ is an equivalence,
  \item 1-cells $(p,p',\al)\cn (b,b', f) \to (c,c',g)$ where
    $p \cn b \to c$, $p' \cn b' \to c'$, and $\al \cn p'f \cong gp$, and
  \item 2-cells $(\Ga,\Ga') \cn (p,p',\al) \Rightarrow (q,q',\be)$
    where $\Ga \cn p \Rightarrow q$, $\Ga' \cn p' \Rightarrow q'$ commuting with
    $\al, \be$ in the obvious way.
  \end{itemize}
  Tedious calculation shows that if $\cB$ is a permutative
  Gray-monoid, then so is $\cB^I$.  There are strict projection
  functors to each coordinate which are both symmetric monoidal
  biequivalences $e_0, e_1 \cn \cB^I \to \cB$, and the inclusion
  $i \cn \cB \to \cB^I$ which sends $b$ to $(b,b,\id_b)$ which is also
  a strict symmetric monoidal biequivalence if $\cB$ is a permutative
  Gray-monoid. Since $e_0 i = e_1 i = \id_{\cB}$, we get that
  $e_0 = e_1$ in $\Ho(\PGM, \cateq)$. Finally, given symmetric
  monoidal pseudofunctors $F,G \cn \cA \to \cB$ and a symmetric
  monoidal equivalence $\al$ between them, 
  we have a symmetric monoidal
  pseudofunctor $A_{\al}: \cA \to \cB^I$ such that $e_0 A_{\al} = F$,
  $e_1 A_{\al} = G$, and the values of $A_\al$ on the nontrivial
  morphism in $I$ are given by the components of $\al$.
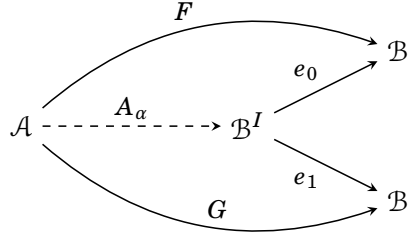
  \[
    \begin{tikzpicture}[x=20mm,y=10mm]
      \draw[tikzob,mm] 
      (1,1) node (c0) {\cB}
      (1,-1) node (c1) {\cB}
      (-1.5,0) node (b) {\cA}
      (0,0) node (ci) {\cB^I}
      ;
      \path[tikzar,mm] 
      (b) edge[bend left] node {F} (c0)
      (b) edge[bend right] node {G} (c1)
      (ci) edge node {e_0} (c0)
      (ci) edge[swap] node {e_1} (c1)
      (b) edge[dashed] node {A_\al} (ci)
      ;
    \end{tikzpicture}
  \]
  
  Given a symmetric monoidal equivalence $\al$ between $F$ and $G$, we
  now apply the observation in \cref{eq:Phi-SF} to the pseudofunctor $A_\al$ and deduce that
  the following holds in $\Ho(\PGM, \cateq)$:
  \[
    \begin{array}{rcl}
      \nu F^{qst} \nu^{-1} & = & \nu \, (e_0 A_{\al})^{qst}\,  \nu^{-1} \\
                           & = & e_0 \nu A_{\al}^{qst} \nu^{-1} \\
                           & = & e_1 \nu A_{\al}^{qst} \nu^{-1} \\
                           & = &  \nu (e_1 A_{\al})^{qst} \nu^{-1} \\
                           & = & \nu G^{qst} \nu^{-1}.
    \end{array}
  \] 
  Thus if two symmetric monoidal pseudofunctors are equivalent, then
  their images in $\Ho(\PGM, \cateq)$ are equal and therefore
  $\Phi_{\cA,\cB}$ is well-defined.  This implies that $\Phi$ is
  functorial because $(GF)^{qst}$ is equivalent to $G^{qst}F^{qst}$
  once again by \cref{lem:qst-detail}.  This finishes the construction
  of a functor $\Ho \mathsf{PGM} \to \Ho(\PGM, \cateq)$.

  A functor in the other direction is even easier to construct.  Once
  again we take the function on objects to be the identity.  For a
  strict functor $F$, we take the image to be $F$ considered as a
  symmetric monoidal pseudofunctor; for the formal inverse of a strict
  biequivalence $F$, we take the image to be the equivalence class of
  weak inverses for $F$ in $\mathsf{PGM}$.  Note that all weak
  inverses are equivalent, so this is well-defined up to
  equivalence. Now $\nu F^{qst} = F\nu$ in $\PGM$, again by naturality
  of $\nu$ with respect to strict functors, so if $G$ is a weak
  inverse for $F$ in $\mathsf{PGM}$ then
  \[
    F\nu G^{qst} \nu^{-1}  = \nu F^{qst}G^{qst} \nu^{-1} 
    = \nu (FG)^{qst} \nu^{-1} = \nu \nu^{-1} = \id_{\cA}
  \]
  in $\Ho(\PGM, \cateq)$ and $\nu G^{qst} \nu^{-1}$ is also
  the formal inverse for $F$.  This calculation shows that the
  composite function
  \[
    \{\cA, \cB\} \to [\cA, \cB] \to \{\cA, \cB\}
  \]
  is the identity. The composite
  \[
    [\cA, \cB] \to \{\cA, \cB\} \to [\cA, \cB]
  \]
  is also the identity since $F$ is equivalent to $\nu F^{qst} \nu^{-1}$
  for any symmetric monoidal pseudofunctor $F$ by \cref{thm:cohqs2cats2} (ii).
\end{proof}

\begin{rmk} 
  The construction of the path object $\cB^I$ is mentioned in
  \cite[Remark 2.69]{SP2011Classification} where it is claimed that
  this path object can be used to equip the category of symmetric
  monoidal bicategories and strict functors with a transferred Quillen
  model structure. As far as we are aware, this claim is incorrect, as
  the map $\cB \to \cB^I$ is not strict without further restrictions
  on $\cB$. For example, this inclusion will not strictly preserve the
  monoidal structure unless $\id_a \oplus \id_b = \id_{a \oplus b}$
  which does not hold in a generic symmetric monoidal bicategory.
\end{rmk}

\begin{cor}\label{cor:smbs=smbps}
  There is an equivalence of categories between
  $\Ho (\SMBicat_{s}, \cateq)$ and $\Ho\SMB$.
\end{cor}
\begin{proof}
  In both cases, every object is isomorphic to a permutative
  Gray-monoid, and thus we have equivalences for the vertical
  inclusions of subcategories in the following display.  
  \[
    \begin{tikzpicture}[x=40mm,y=20mm]
      \draw[tikzob,mm] 
      (0,0) node (a) {\Ho \SMB}
      (1,0) node (b) {\Ho ( \SMBicat_{s}, \cateq)}
      (0,1) node (c) {\Ho\mathsf{PGM}}
      (1,1) node (d) {\Ho ( \PGM, \cateq)}
      ;
      \path[tikzar,mm] 
      (c) edge[swap] node {\hty} (a)
      (d) edge node {\hty} (b)
      (a) edge[dashed] node {} (b)
      (c) edge node {\Phi} (d)
      ;
    \end{tikzpicture}
  \]
  The isomorphism $\Phi$ from \cref{prop:smbs=smbps} therefore induces
  an equivalence for the dashed arrow shown above.
\end{proof}

\begin{lem}\label{lem:PGM-to-SMB-equiv}
  The inclusions
  \[
    \PGM \hookrightarrow \SMBicat_{s}, \quad \Pic\PGM \hookrightarrow
    \Pic\Bicat_{s}
  \]
  induce equivalences of homotopy theories
  \[
    (\PGM, \cW \cap \PGM) \hookrightarrow (\SMBicat_{s}, \cW), 
  \]
  \[
    (\Pic\PGM, \cW \cap \Pic\PGM) \hookrightarrow (\Pic\Bicat_{s}, \cW)
  \]
  for any class of morphisms $\cW$ that includes all biequivalences.
\end{lem}
\begin{proof}\proofof{lem:PGM-to-SMB-equiv}\marginref{cohqs2cats,thm:cohqs2cats2}
  By \cref{cohqs2cats,thm:cohqs2cats2}, we have natural
  strict biequivalences
  \[
  \cA \longleftarrow \cA^{c} \longrightarrow \cA^{qst}
  \]
  for any symmetric monoidal bicategory $\cA$, and a natural strict
  biequivalence $\cB^{qst} \to \cB$ for any permutative Gray-monoid
  $\cB$.  This implies that the inclusion and quasistrictification
  induce weak equivalences between the Rezk nerves of
  $(\PGM, \cW \cap \PGM)$ and $(\SMBicat_{s}, \cW)$, and hence give an
  equivalence of homotopy theories (see, e.g., \cite[Corollary
  2.9]{GJO2017KTheory}).  The same argument implies the final
  equivalence of homotopy theories because the property of being
  Picard is invariant under biequivalences.
\end{proof}

The most important case of such a $\cW$ is the class of
$P_2$-equivalences we define now.

\begin{defn}\label{defn:p2-eq-SMB}
  A functor (of any type) $F\cn \cA \to \cB$ of bicategories is a
  \emph{$P_2$-equivalence} if the induced map of topological 
	spaces $NF\cn N\cA
  \to N\cB$ is a $P_2$-equivalence, i.e., induces an isomorphism
	on $\pi_n$ for $n=0, 1, 2$ and all choices of basepoint.
\end{defn}

\section{Symmetric monoidal structures from operad actions}
\label{sec:operads}

In this section we describe how to extract symmetric monoidal
structure from an operad action on a 2-category.  We describe the
motivating example of algebras over the Barratt-Eccles operad, but
then abstract the essential features to a general theory.  Our main
applications appear in \cref{sec:gp-cpn}, where we use this theory and
the topological group-completion theorem of May \cite{May1974Einfty}
to deduce information about the fundamental 2-groupoid of a
group-completion.

\subsection{Background about operads}
\label{sec:background}

\begin{defn}
  Let $(\sV,\star,e)$ be a symmetric monoidal category. An
  \emph{operad} $P$ in $\sV$ is a sequence $\{P(n)\}_{n\geq 0}$ of
  objects in $\sV$ such that $P(n)$ has a (right) $\Si_n$-action,
  together with morphisms
  \[
  \ga \cn P(n) \star P(k_1) \star \cdots \star P(k_n) \rtarr 
  P(k_1 + \cdots +k_n)
  \]
  and 
  \[
  \idop \cn e \rtarr P(1)
  \]
  that are appropriately equivariant and compatible.
  See \cite{may72geo} or \cite{Yau2016Colored} for a complete description.
    
  A \emph{map} of operads $f\cn P \rtarr Q$ is given by a $\Si_n$-map
  $f_n\cn P(n) \rtarr Q(n)$ for each $n\geq 0$ compatible with the
  operations and the identity.
    
  A \emph{$P$-algebra} is given by a pair $(X,\mu)$, where $X$ is an
  object of $\sV$, and $\mu$ is a collection of morphisms
  \[
  \mu_n \cn P(n) \star X^{\star n} \rtarr X
  \]
  in $\sV$ that are appropriately equivariant and compatible with
  $\ga$ and $\idop$. A morphism of $P$-algebras
  $(X,\mu) \rtarr (X',\mu')$ is given by a morphism $g\cn X \rtarr X'$
  compatible with the maps $\mu_n$ and $\mu'_n$.
\end{defn}    

In this paper we will be concerned with operads in
$(\sset, \times,*)$ and $(\Top,\times,*)$, as well as several variants
for 2-categories, including $(\iicat,\times,*)$, $(\iicat,\otimes,*)$,
and $(\iicatps,\times,*)$.

\begin{notn}
  We let $\PAlg(\sV,\star)$
 denote the category of $P$-algebras in
  $(\sV, \star)$ and their morphisms.\footnote{We do not include the
    monoidal unit in the notation as in all of our cases it will be
    the terminal object.} If there is no confusion over the ambient
  symmetric monoidal category we also write $\PAlg$.
\end{notn}

We now recall and fix notation for standard transfers of operadic
structures; see e.g., \cite{JY2015Foundation,Yau2016Colored}.
\begin{lem}
  A map of operads $f\cn P \rtarr Q$ induces a functor
  \[
  f^* \cn Q\mh\mb{Alg} \rtarr \PAlg.
  \]
  that sends a $Q$-algebra $(X,\mu)$ to the $P$-algebra $(X,\nu)$,
  where $\nu_n$ is given by the composite
  \[
    P(n)\star X^{\star n} \fto{f_n\star \id} Q(n)\star X^{\star n}
    \fto{\mu_n} X.
  \]
  If $f$ and $g$ are composable maps of operads, then $(g\circ f)^*=f^*\circ g^*$.
\end{lem}

\begin{lem}\label{lem:operad_monoidal} 
  Let $P$ be an operad in $(\sV,\star)$, and let
  $F,G\cn (\sV,\star) \rtarr (\sW,\lozenge)$ be lax symmetric monoidal
  functors, and $\al \cn F \Rightarrow G$ a monoidal natural
  transformation. Then
  \begin{enumerate}
  \item $FP$ is an operad in $(\sW,\lozenge)$;
  \item $F$ induces a functor
    \[
      F\cn \PAlg(\sV,\star) \rtarr F\PAlg(\sW,\lozenge)
    \]
    that sends $(X,\mu)$ to $(FX,\nu)$, where $\nu_n$ is given by the
    composite
    \[
      FP\lozenge (FX)^{\lozenge n} \rtarr F(P\star X^{\star n}) \fto{F\mu_n} FX;
    \]
    and
  \item $\al$ induces a map of operads $\al\cn FP \rtarr GP$, and a
    natural transformation $\al\cn F \Rightarrow \al^*\circ G$ of
    functors
    \[
      \PAlg(\sV,\star) \rtarr F\PAlg(\sW,\lozenge),
    \]
    whose component at $(X,\mu)$ is given by $\al_X$.
  \end{enumerate}
\end{lem}

\begin{cor}\label{cor:operads-2cat-variants}
  Any operad $P$ in $(\iicat,\times)$ gives rise to an operad (with
  the same underlying sequence of objects) in $(\iicat, \otimes)$ and
  $(\iicatps,\times)$.

  There are inclusions
  \[
  \PAlg(\iicat,\times) \rtarr \PAlg(\iicat,\otimes) \rtarr
  \PAlg(\iicatps,\times).
  \]
\end{cor}
\begin{proof}\proofof{cor:operads-2cat-variants}\marginref{lem:operad_monoidal} 
  Following \cref{rem:id-is-lax-monoidal}, we can use the identity
  functor $(\iicat, \times) \rtarr (\iicat, \otimes)$ and the
  inclusion $(\iicat,\otimes) \rtarr (\iicatps,\times)$ to transfer
  the algebra structures.
\end{proof}

\begin{rmk}
  Given the relationship between cubical functors and the Gray tensor
  product, we can specify the objects in the categories of algebras
  above as follows. In all cases, an algebra is given by a pair
  $(X,\mu)$, where $X$ is a 2-category, and $\mu$ is a collection of
  morphisms $\mu_n \cn P(n) \times X^n \rtarr X$, which are 2-functors
  if working in $(\iicat,\times)$, cubical functors if working in
  $(\iicat,\otimes)$, or pseudofunctors if working in
  $(\iicatps,\times)$.
  In the first two cases, a map of algebras is given by a 2-functor
  $X \rtarr X'$ commuting with $\mu$, but for the latter case, a map
  of algebras is given by a pseudofunctor $X\rtarr X'$.
\end{rmk}

\begin{notn}\label{notn:P-2cat}
  Let $P$ be an operad in $(\iicat, \times)$.  Let $P\mh\iicat$ denote
  the subcategory of $\PAlg(\iicatps,\times)$ given by all objects and
  those morphisms whose underlying pseudofunctor $X\rtarr X'$ is a 2-functor.
\end{notn}

\begin{defn}[The Barratt-Eccles operad]\label{defn:operadO}
  Let $\mathcal{O}$ denote the operad in $(\cat, \times)$ with
  $\mathcal{O}(n)$ being the translation category for $\Sigma_n$: the
  objects of $\mathcal{O}(n)$ are the elements of the symmetric group
  $\Sigma_n$, and there is a unique isomorphism between any two
  objects. By abuse of notation, we also denote by $\cO$ the operad in
  $(\iicat, \times)$ obtained by adding identity 2-cells. Note that
  $\cO$ is also an operad in $(\iicat,\otimes)$ and
  $(\iicatps,\times)$.
  
  Because the nerve and geometric realization functors are strong
  monoidal, we have an operad $|N\cO| = B\cO$ in $\Top$.  Likewise, if
  $\cA$ is an $\cO$-algebra in ($\iicatps$,$\times$), then $B\cA$ is a
  $B\cO$-algebra in $\Top$.
  \marginproof{defn:operadO}{notn:B,prop:classical-monoidal-functors}
\end{defn}

\begin{rmk}\label{rmk:peter-invented-o}
  The operad $B\cO$ in $\Top$ was used implicitly in
  \cite{Bar1971FreeGroup}.  As an operad in $\cat$, $\cO$ was
  independently introduced by May in \cite{may72geo,May1974Einfty}. In
  \cite{May1974Einfty}, May shows that a permutative category is an
  $\cO$-algebra [\textit{loc.\ cit.} Lemmas 4.3-4.5] and that $B\cO$
  is an $E_\infty$ operad [\textit{loc.\ cit.} Lemma 4.8].
\end{rmk}

\subsection{Symmetric monoidal structures from arbitrary operads}
\label{sec:choice-of-mult}

The relationship between $\cO$-algebras in $\Cat$ and permutative
categories extends to the 2-categorical level.  We describe this now,
and then abstract the key features for general operads.  To begin,
note that an operad $P$ in $(\iicat, \times)$ induces a different
monad $P_{\otimes}$ on $\iicat$
using a combination of the cartesian
product and the Gray tensor product. Explicitly, $P_\otimes$ is
defined via the formula
\[
X \mapsto \coprod_{n\geq 0} P(n) \times_{\Sigma_n} X^{\otimes n}
\]
using the natural 2-functor
\[
(\cA \times \cB) \otimes (\cA' \times \cB') \to (\cA \otimes \cA')
\times (\cB \otimes \cB') \to (\cA \times \cA') \times (\cB \times
\cB').
\]

\begin{prop}\label{prop:OAlg-gray-is-PGM}
  The category $\cO_{\otimes} Alg$ of algebras for the monad
  $\cO_{\otimes}$ on $\iicat$ is isomorphic to $\PGM$.
\end{prop}
\begin{proof}\proofof{prop:OAlg-gray-is-PGM}\marginref{cor:operads-2cat-variants}
  Note that objects $e \in \cA$ are in bijection with 2-functors
  $\cO(0) \to \cA$, and that 2-functors
  $\oplus \cn \cA \otimes \cA \to \cA$ together with a 2-natural
  isomorphism
  \[
  \begin{tikzpicture}[x=1mm,y=1mm]
    \draw[tikzob,mm] 
    (0,0) node (00) {\cA \otimes \cA}
    (25,0) node (10) {\cA \otimes \cA}
    (12.5,-10) node (01) {\cA}
    ;
    \path[tikzar,mm] 
    (00) edge node{\tau} (10)
    (10) edge node{\oplus} (01)
    (00) edge[swap] node{\oplus} (01)
    ;
    \draw[tikzob,mm] 
    (12.5,-4) node {\Anglearrow{40} \beta}
    ;
  \end{tikzpicture}
  \]
  where $\tau \cn \cA \otimes \cA \to \cA \otimes \cA$ is the symmetry
  isomorphism in $\IICat$ for the Gray tensor product are in bijection
  with 2-functors
  \[
    \tilde{\oplus} \cn \cO(2) \times_{\Sigma_2} \cA^{\otimes 2} \to \cA
  \]
  using the strict parts of \cref{lem:psfun-out-of-product}. It is now
  straightforward to verify that the axioms for a permutative
  Gray-monoid are the same as those for an algebra over
  $\cO_{\otimes}$.
\end{proof}

Using \cref{cor:operads-2cat-variants} with
\cref{prop:OAlg-gray-is-PGM}, we can also regard a permutative
Gray-monoid as an algebra with respect to the cartesian product.  We
will use this implicitly in our work below.
\begin{cor}\label{prop:PGM-is-OAlg}
  Every permutative Gray-monoid is an algebra for the operad $\cO$
  acting on $(\iicatps,\times)$.
\end{cor} 
\marginproof{prop:PGM-is-OAlg}{cor:operads-2cat-variants,prop:OAlg-gray-is-PGM}

Although we have no use for it here, there is an analogous result for
$\cO$-algebras in $(\iicat,\times)$ and the stricter notion of
\emph{permutative 2-category} described in \cite{GJO2017KTheory},
which we state in the following proposition.  This is the
$\cat$-enriched version of the statement that permutative
\emph{categories} are precisely the $\cO$-algebras in $(\cat,\times)$
(see \cref{rmk:peter-invented-o}).

\begin{prop}\label{prop:OAlg-is-P2Cat}
  The category of $\cO$-algebras in $(\iicat, \times)$ is
  isomorphic to the category $\piicat$ of permutative 2-categories and
  strict functors of such.
\end{prop}

We now turn to the general question of how symmetric monoidal
structures arise from operad actions on 2-categories.
\begin{defn}
  Let $\mathbb{P}$ be a property of 2-categories.  We write
  $\mathbb{P}(\leq n)$ (including the case $n = \infty$) for the full
  subcategory of the category of operads consisting of those operads
  $P$ for which $P(k)$ has $\mathbb{P}$ for all $k \leq n$.
\end{defn}

\begin{notn}
  Let $\bC$ denote the property of being bicategorically contractible,
  i.e., $X$ has $\bC$ if the unique 2-functor $X \to *$ is a
  biequivalence.
\end{notn}

\begin{lem}\label{lem:contractible-iff}
  A nonempty 2-category $X$ is contractible if and only if the
  following four conditions hold.
  \begin{enumerate}
  \item Any two objects are connected by a 1-cell.
  \item Every 1-cell is an equivalence.
  \item Every 2-cell is invertible.
  \item Any two parallel 1-cells are connected by a unique
    2-isomorphism.
  \end{enumerate}
\end{lem}

\begin{example}
  The operad $\mathcal{O}$ of \cref{defn:operadO} is in
  $\bC(\leq \infty)$.
\end{example}

\begin{defn}
  A \emph{choice of multiplication} $\chi$ in $P$ consists of the
  following:
  \begin{itemize}
  \item a choice of an object $i \in P(0)$;
  \item a choice of an object $t \in P(2)$;
  \item an adjoint equivalence
    $\al\colon\ga(t; t, \idop) \simeq \ga(t; \idop, t)$ in $P(3)$;
  \item adjoint equivalences
    $l\colon \ga(t; i, \idop) \simeq \idop$ and $r\colon \idop \simeq
    \ga(t; \idop, i)$ in $P(1)$;
  \item invertible 2-cells $\pi$ in $P(4)$, $\mu$ in $P(3)$, and
    $\lambda$ and $\rho$ in $P(2)$, as depicted below;
  \item an adjoint equivalence $\be\cn t \simeq t \cdot (12)$ in
    $P(2)$;\item invertible 2-cells $R_{-|--}$ and $R_{--|-}$ in
    $P(3)$, and $v$ in $P(2)$, as depicted below.
  \end{itemize}
  
   \[
  \begin{tikzpicture}[x=25mm,y=20mm]
    \draw[tikzob,mm] 
    (0,0) node (0) {\ga(t;\ga(t;t,\idop),\idop)}
    (0,-.4) node (0') {\ga(\ga(t;t,\idop);t,\idop,\idop)}
    (0,-.2) node[rotate=90] {=}
    (1,1) node (1) {\ga(\ga(t;t,\idop);\idop,t,\idop)}
    (1,.6) node (1') {\ga(t;\ga(t;\idop,t),\idop)}
    (1,.8) node[rotate=90] {=}
    (3,1) node (2) {\ga(\ga(t;\idop,t);\idop,t,\idop)}
    (3,.6) node (2') {\ga(t;\idop,\ga(t;t,\idop))}
    (3,.8) node[rotate=90] {=}
    (4,0) node (3) {\ga(t;\idop,\ga(t;\idop,t))}
    (4,-.4) node (3') {\ga(\ga(t;\idop,t);\idop,\idop,t)}
    (4,-.2) node[rotate=90] {=}
    (2,-1) node (4) {\ga(\ga(t;\idop,t);t,\idop,\idop)}
    (2,-1.4) node (4') {\ga(\ga(t;t,\idop);\idop,\idop,t)}
    (2,-1.2) node[rotate=90] {=};
    \path[mm,auto,->] 
    (0) edge node {\ga(\id;\al,\id)} (1')
    (1) edge node {\ga(\al;\id,\id)} (2)
    (2') edge node {\ga(\id;\id,\al)} (3)
    (0') edge[swap] node {\ga(\al;\id;\id,\id)} (4)
    (4'.15) edge[swap] node {\ga(\al;\id,\id,\id)} (3');
    \draw[tikzob,mm]
    (2,0) node[rotate=270,font=\Large] {\Rightarrow}
    node[right=.3pc] {\pi};
  \end{tikzpicture}
  \]

  \[
  \begin{tikzpicture}[x=55mm,y=20mm]
    \draw[tikzob,mm] 
    (0,0) node (0) {\ga(\ga(t;t,\idop);\idop,i,\idop)}
    (0,-.4) node (0') {\ga(t;\ga(t;\idop,i),\idop)}
    (0,-.2) node[rotate=90] {=}
    (1,0) node (1) {\ga(\ga(t;\idop,t);\idop,i,\idop)}
    (1,-.4) node (1') {\ga(t;\idop,\ga(t;i,\idop))}
    (1,-.2) node[rotate=90] {=}
    (-.5,-1.3) node (2) {\ga(t;\idop,\idop)}
    (1.5,-1.3) node (3) {\ga(t;\idop,\idop)};
    \path[tikzar,mm] 
    (0) edge node {\ga(\al;\id,\id,\id)} (1)
    (2) edge node {\ga(\id;r,\id)} (0')
    (2) edge[swap] node {\id} (3)
    (1') edge node {\ga(\id;\id,l)} (3);
    \draw[tikzob,mm]
    (0.5,-0.65) node[rotate=270,font=\Large] {\Rightarrow}
    node[right=.3pc] {\mu};
  \end{tikzpicture}
  \]
  
  \[
  \begin{tikzpicture}[x=60mm,y=20mm]
    \draw[tikzob,mm] 
    (0,0) node (0) {\ga(t;\ga(t;i,\idop),\idop)}
    (0,-.4) node (0') {\ga(\ga(t;t,\idop);i,\idop,\idop)}
    (0,-.2) node[rotate=90] {=}
    (1,0) node (1) {\ga(t;\idop,\idop)}
    (1,-.4) node (1') {\ga(\idop;t)}
    (1,-.2) node[rotate=90] {=}
    (.5,-1) node (2) {\ga(\ga(t;i,\idop);t)}
    (.5,-1.4) node (2') {\ga(\ga(t;\idop,t);i,\idop,\idop)}
    (.5,-1.2) node[rotate=90] {=};
    \path[tikzar,mm] 
    (0) edge node {\ga(\id;l,\id)} (1)
    (0') edge[swap] node {\ga(\al;\id,\id,\id)} (2'.165)
    (2) edge[swap] node {\ga(l,\id)} (1');
    \draw[tikzob,mm]
    (0.5,-0.5) node[rotate=270,font=\Large] {\Rightarrow}
    node[right=.3pc] {\la};
  \end{tikzpicture}
  \]

  \[
  \begin{tikzpicture}[x=60mm,y=20mm]
    \draw[tikzob,mm] 
    (0,0) node (0) {\ga(t;\idop,\idop)}
    (0,-.4) node (0') {\ga(\idop;t)}
    (0,-.2) node[rotate=90] {=}
    (1,0) node (1) {\ga(t;\idop,\ga(t;\idop,i))}
    (1,-.4) node (1') {\ga(\ga(t;\idop,t);\idop,\idop,i)}
    (1,-.2) node[rotate=90] {=}
    (.5,-1) node (2) {\ga(\ga(t;\idop,i);t)}
    (.5,-1.4) node (2') {\ga(\ga(t;t,\idop);\idop,\idop,i)}
    (.5,-1.2) node[rotate=90] {=};
    \path[tikzar,mm] 
    ;
\path[mm,auto,->] 
    (0) edge node {\ga(\id;\id,r)} (1)
    (0') edge[swap] node {\ga(r,\id)} (2)
    (2'.15) edge[swap] node {\ga(\al;\id,\id,\id)} (1');
    \draw[tikzob,mm]
    (0.5,-0.5) node[rotate=270,font=\Large] {\Rightarrow}
    node[right=.3pc] {\rho};
  \end{tikzpicture}
  \]
  
  \[
  \begin{tikzpicture}[x=25mm,y=20mm]
    \draw[tikzob,mm] 
    (0,-.2) node (0) {\ga(t;t,\idop)}
    (1,1) node (1) {\ga(t;t,\idop)\cdot(12)}
    (1,.6) node (1') {\ga(t;t\cdot(12),\idop)}
    (1,.8) node[rotate=90] {=}
    (3,1) node (2) {\ga(t;\idop,t)\cdot(12)}
    (4,0) node (3) {{\ga(t;\idop,t\cdot(12))\cdot(12)}}
    (4,-.4) node (3') {{\ga(t;\idop,t)\cdot(132)}}
    (4,-.2) node[rotate=90] {=}
    (1,-1.4) node (4) {\ga(t;\idop,t)}
    (3,-1) node (5) {\ga(t;t,\idop)\cdot(132)}
    (3,-1.4) node (5') {\ga(t\cdot(12);\idop,t)}
    (3,-1.2) node[rotate=90] {=};
    \path[tikzar,mm] 
    (0) edge node {\ga(\id;\be,\id)} (1')
    (1) edge node {\al\cdot(12)} (2)
    (2) edge node {\ga(\id;\id,\be)\cdot(12)} (3)
    (0) edge[swap] node {\al} (4)
    (4) edge[swap] node {\ga(\be;\id,\id)} (5')
    (5) edge[swap] node {\al\cdot(132)} (3');
    \draw[tikzob,mm]
    (2,-.2) node[rotate=270,font=\Large] {\Rightarrow}
    node[right=.3pc] {R_{-|--}};
  \end{tikzpicture}
  \]

  \[
  \begin{tikzpicture}[x=25mm,y=20mm]
    \draw[tikzob,mm] 
    (0,-.2) node (0) {\ga(t;\idop,t)}
    (1,1) node (1) {\ga(t;\idop,t)\cdot(23)}
    (1,.6) node (1') {\ga(t;\idop,t\cdot(12))}
    (1,.8) node[rotate=90] {=}
    (3,1) node (2) {\ga(t;t,\idop)\cdot(23)}
    (4,0) node (3) {{\ga(t;t\cdot(12),\idop)\cdot(23)}}
    (4,-.4) node (3') {{\ga(t;t,\idop)\cdot(123)}}
    (4,-.2) node[rotate=90] {=}
    (1,-1.4) node (4) {\ga(t;t,\idop)}
    (3,-1) node (5) {\ga(t;\idop,t)\cdot(123)}
    (3,-1.4) node (5') {\ga(t\cdot(12);t,\idop)}
    (3,-1.2) node[rotate=90] {=};
    \path[tikzar,mm] 
    (0) edge node {\ga(\id;\id,\be)} (1')
    (1) edge node {\al^\bullet\cdot(23)} (2)
    (2) edge node {\ga(\id;\be,\id)\cdot(23)} (3)
    (0) edge[swap] node {\al^\bullet} (4)
    (4) edge[swap] node {\ga(\be;\id,\id)} (5')
    (5) edge[swap] node {\al^\bullet\cdot(123)} (3');
    \draw[tikzob,mm]
    (2,-.2) node[rotate=270,font=\Large] {\Rightarrow}
     node[right=.3pc] {R_{--|-}};
  \end{tikzpicture}
  \]
  
  \[
  \begin{tikzpicture}[x=60mm,y=20mm]
    \draw[tikzob,mm] 
    (0,-1) node (0) {t}
    (1,-1) node (1) {t}
    (.5,0) node (2) {t\cdot(12)};
    \path[tikzar,mm] 
    (0) edge[swap] node {\id} (1)
    (0) edge node {\be} (2)
    (2) edge node {\be\cdot(12)} (1);
    \draw[tikzob,mm]
    (0.5,-0.5) node[rotate=270,font=\Large] {\Rightarrow}
    node[right=.3pc] {v};
  \end{tikzpicture}
  \]
  
  Note that in all of the diagrams above, the equalities on objects
  follow from the axioms of an operad.  
\end{defn}

\begin{rmk}
  Our choice of multiplication is reminiscent of Batanin's notion of a
  system of compositions on a globular operad
  \cite{Bat1998Monoidal}. In both instances, this extra structure on
  an operad is intended to pick out preferred binary operations,
  ensuring they exist as needed. \Cref{prop:C4-choice-of-mult} below
  also reflects the modification made by Leinster \cite{Lei04Higher}
  in which contractibility of the operad already ensures enough
  operations.
\end{rmk} 
  
\begin{example}\label{ex:choiceO}
  Consider the Barratt-Eccles operad $\cO$ of
  \cref{defn:operadO}. There is a canonical choice of multiplication
  $\can$ on $\cO$ given by $i=\ast \in \cO(0)$ and $t$ equal to the
  identity permutation in $\Si_2$. All the equivalences are given by
  the unique 1-morphisms between the corresponding objects, and all
  the 2-cells are the identity (noting that the boundaries are equal
  because there is a unique 1-morphism between any two objects).
\end{example}
 
The idea behind this example can be generalized to a larger class of
operads, and we explain this now.

\begin{prop}\label{prop:C4-choice-of-mult}
  Let $P$ be an operad in $\bC(\leq 4)$. Then there exists a choice of
  multiplication on $P$.
\end{prop} 

\begin{proof}\proofof{prop:C4-choice-of-mult}
  Since $P(0)$ and $P(2)$ are contractible, they are in particular
  non-empty, and hence we can pick objects $i\in P(0)$ and
  $t\in P(2)$.  By \cref{lem:contractible-iff} the equivalences $\al$,
  $l$, $r$ and $\be$ and the 2-isomorphisms $\pi$, $\mu$, $\la$,
  $\rho$, $R_{-|--}$, $R_{--|-}$ and $v$ can be picked using
  contractibility. \end{proof}
 
\begin{prop}\label{prop:operad-map-induces-mult-choice}
  Let $f \cn P \to Q$ be a map of operads in $(\iicatps,\times)$.  If
  $P$ has a choice of multiplication $\chi$, then taking the images of
  all data involved gives a choice of multiplication $f(\chi)$ in
  $Q$.
\end{prop}

\begin{thm}\label{thm:choice-of-mult-to-SMB}
  Let $P$ be an operad in $\bC(\leq 5)$.  Then a choice of
  multiplication $\chi$ in $P$ determines a functor 
  \[
  \chi^* \cn \PAlg(\iicatps,\times) \to \SMIICat_{ps}
  \]
  which is the identity on underlying 2-categories and
  pseudofunctors.  We therefore have a functor
  \[
  \chi^* \cn P\mh\IICat \to \SMIICat_{s}.
  \]
\end{thm}
\begin{proof}\proofof{thm:choice-of-mult-to-SMB}
  First, observe that the second statement is a refinement of the
  first because $P\mh\IICat$ denotes the subcategory of
  $\PAlg(\iicatps,\times)$ whose morphisms are 2-functors (see
  \cref{notn:P-2cat}).  To prove the first statement, we use $\chi$ to
  construct a symmetric monoidal structure on an arbitrary $P$-algebra
  $(X,\mu)$.  

  We define the monoidal product as the
  pseudofunctor $\mu_2(t;-,-)$ obtained by applying
  \cref{lem:psfun-out-of-product} to the pseudofunctor
  $\mu_2\cn P(2) \times X^2 \to X$. Explicitly, the monoidal product
  on objects is defined as $xy=\mu_2(t;x,y)$. The unit object is
  defined as $e=\mu_0(i)$.
 
  The equivalences $\al$, $l$, $r$ and $\be$ are defined as the
  appropriate images of their namesakes in $\chi$, as given by
  \cref{lem:psfun-out-of-product}. For example,
  $\al\cn (xy)z \to x(yz)$ is given by the pseudonatural equivalence
  $\mu_3(\al;-,-,-)$. Similarly, the invertible modifications $\pi$,
  $\mu$, $\la$, $\rho$, $R_{-|--}$, $R_{--|-}$ and $v$ are given by
  the appropriate images of their namesakes in $\chi$. For example,
  $v$ is defined as $\mu_2(v;-,-)$.
 
  Contractibility of $P(n)$ for $n\leq 5$ implies that analogues of
  the axioms for the modifications in a symmetric monoidal bicategory
  are satisfied by the corresponding cells in $P$, which in turn
  implies that the same axioms are satisfied after applying $\mu$.
 
  Let $F\cn X \to Y$ be a $P$-algebra morphism, i.e., a pseudofunctor
  that commutes with the action of $P$. It is easy to check that this
  implies that $F$ preserves strictly the symmetric monoidal
  structures on $X$ and $Y$ given by the choice of multiplication
  $\chi$.  Thus we have a functor
  $\chi^*\cn \PAlg(\iicatps,\times) \to \SMIICat_{ps}$.
\end{proof}

We now record several results which follow from
\cref{thm:choice-of-mult-to-SMB} and its proof.
\begin{prop}\label{prop:chistar-fstar-equals-chifstar}
  Let $f \cn P \to Q$ in $\bC(\leq 5)$, and assume $P$ has choice of
  multiplication $\chi$. As functors
  $Q\mh\IICat \to \SMIICat_{s}$, we have an
  equality $\chi^* \circ f^* = (f(\chi))^*$.
\end{prop}

\begin{prop}\label{prop:choice-from-product}
  Suppose that $P_1$ and $P_2$ are operads with choices of
  multiplication $\chi_1$ and $\chi_2$.  Then the product $P = P_1
  \times P_2$ has a choice of multiplication defined by the pointwise
  product of the data, $\chi = \chi_1 \times \chi_2$.
\end{prop}

\begin{rmk}\label{rmk:product-choice-projections}
  As a special case of \cref{prop:chistar-fstar-equals-chifstar}, the
  projections $\pi_i\cn P \to P_i$ identify $\pi_i(\chi)$ with
  $\chi_i$.
\end{rmk}

\begin{prop}
  If $P$ is an operad in $\bC(\leq 5)$, $X$ is a $P$-algebra, and
  $\chi_1, \chi_2$ are two different choices of multiplication in $P$,
  the identity pseudofunctor on $X$ can be equipped with the structure of a
  symmetric monoidal biequivalence $\chi_1^\ast X \to \chi_2^\ast X$
  relating the two symmetric monoidal structures. Moreover, this
  assignment is natural in $X$.
\end{prop}
\begin{proof}
  Note that contractibility
  of $P$ gives 1-morphisms relating $i_1$ with $i_2$ and $t_1$ with
  $t_2$, which when applied to the algebra $X$ give rise to the map
  that compares units and multiplications. The rest of the data and
  axioms of a symmetric monoidal biequivalence follow from
  contractibility as well.
\end{proof}

Recall that $\can$ denotes the canonical choice of multiplication for
the Barratt-Eccles operad (\cref{ex:choiceO}), and
we implicitly regard a permutative Gray-monoid as an
$\cO$-algebra by \cref{prop:PGM-is-OAlg}.

\begin{prop}\label{prop:PGM-as-O-algebra}
  Let $\cA$ be a permutative Gray-monoid. Then
  $\cA=\can^* \cA$. The composite
  \[
  \PGM \to \cO\mh\mb{Alg}(\iicatps, \times) \fto{\can^*} \SMIICat_{s}
  \]
  is
  the inclusion functor from $\PGM$ to $\SMIICat_{s}$.
  \end{prop}
\begin{proof}\proofof{prop:PGM-as-O-algebra}\marginref{prop:PGM-is-OAlg,thm:choice-of-mult-to-SMB}
  The underlying 2-categories of $\cA$ and $\can^* \cA$
  are equal. It is clear by construction that the two symmetric
  monoidal structures are the same.
\end{proof}

\section{Symmetric monoidal structures and group-completion}
\label{sec:gp-cpn}

In this section we show how to construct a Picard 2-category with the
same stable 2-type as a given permutative Gray-monoid.  Our
construction begins with a strict version of the fundamental
2-groupoid in \cref{sec:W}.  In \cref{sec:K-thy} we analyze its effect
on stable equivalences, and in \cref{sec:top-gp-cpn} we apply the
theory of \cref{sec:operads} together with topological
group-completion to obtain the desired Picard 2-category.

\begin{defn}\label{defn:striigpd}
  A \emph{strict 2-groupoid} is a 2-category in which every 1- and
  2-cell is strictly invertible, i.e., for every 1-cell
  $f \cn a \to b$ there is a 1-cell $g \cn b \to a $ such that $gf$,
  $fg$ are both identity 1-cells, and similarly for 2-cells. We define
  the category $\siigpd$ to have strict 2-groupoids as objects and
  2-functors as morphisms.
\end{defn}

\begin{rmk}
  We should note that any 2-category isomorphic to a strict 2-groupoid
  is itself a strict 2-groupoid, but that a 2-category which is
  biequivalent to a strict 2-groupoid will not in general have 1-cells
  which are strictly invertible in the sense above, but only satisfy
  the weaker condition that we have called invertible in
  \cref{defn:invertible2}: given $f$ there exists a $g$ such that $fg$
  and $gf$ are isomorphic to identity 1-cells.
\end{rmk}

\subsection{Background on the Whitehead 2-groupoid and nerves}
\label{sec:W}

We now recall the construction of a strict fundamental 2-groupoid for
simplicial sets due to Moerdijk-Svensson \cite{MS93Algebraic}, known
as the Whitehead 2-groupoid.
\begin{defn}[{\cite[Section 1, Example (2)]{MS93Algebraic}}]\label{defn:whitehead}
  Let $X$ be a topological space, $Y \subseteq X$ a subspace, and
  $S \subseteq Y $ a subset.  We define the \emph{Whitehead
    2-groupoid} $W(X,Y,S)$ to be the strict 2-groupoid with
  \begin{itemize}
  \item objects the set $S$,
  \item 1-cells $[f] \cn a \to b$ to be homotopy classes of paths $f$
    from $a$ to $b$ in $Y$, relative the endpoints, and
  \item 2-cells $[\al] \cn [f] \Rightarrow [g]$ to be homotopy classes
    of maps $\al \cn I \times I \to X$ such that
    \begin{enumerate}
    \item $\al(t,0) = f(t)$,
    \item $\al(t,1) = g(t)$,
    \item $\al(0,-)$ is constant at the source of $f$ (and hence also
      $g$),
    \item $\al(1,-)$ is constant at the target of $f$ (and hence also
      $g$), and
    \item homotopies $H(s,t,-)$ between two such maps fix the vertical
      sides and map the horizontal sides into $Y$ for each $s$.
    \end{enumerate}
  \end{itemize}
\end{defn}

Since the nerve functor $N \cn \siigpd \to \sset$ preserves limits and
filtered colimits, it has a left adjoint. In \cite{MS93Algebraic},
this left adjoint was explicitly computed using Whitehead 2-groupoids,
and we recall their construction now.

\begin{notn}\label{notn:skeleton}
  For a simplicial set $X$, let $X^{(n)}$ denote the $n$-skeleton of
  $X$.
\end{notn}

\begin{thm}[{\cite[Theorem 2.3]{MS93Algebraic}}]\label{thm:W}
  The functor $W \cn \sset \to \siigpd$, defined by
  \[
    W(X) = W(|X|, |X^{(1)}|, |X^{(0)}|),
  \]
  is left adjoint to the nerve functor, $N$.
\end{thm}

We will need several key properties of $W$ from \cite{MS93Algebraic}; we 
summarize these in the next proposition.
\begin{prop}[{\cite{MS93Algebraic}}]\label{cor:eta_2equiv}
  \
  \begin{enumerate}
  \item\label{it:W-preserves} If $X \to Y$ is a weak equivalence of
    simplicial sets, then $WX \to WY$ is a biequivalence of
    2-groupoids [\textit{loc.~cit.}, Proposition 2.2 (iii)].
  \item\label{it:epz} 
    For a strict 2-groupoid $\cC$, the strict
    2-functor $\epz \cn WN\cC \to \cC$ is a bijective-on-objects
    biequivalence, although its pseudoinverse is only a pseudofunctor
    [\textit{loc.~cit.}, Displays (1.9) and (2.10)].
  \item\label{it:eta} For a simplicial set $X$, the unit 
    $\eta_{X} \cn X \to NWX$ of the adjunction $W \dashv N$ is a
    $P_2$-equivalence.  In particular, if $\cK$ is a 2-category then
    $\eta_{N\cK} \cn N\cK \to NWN\cK$ is a $P_2$-equivalence
    [\textit{loc.~cit.}, Corollary 2.6].
  \item \label{it:W-monoidal} The functor $W$ is strong monoidal with
    respect to the cartesian product [\textit{loc.~cit.}, Proposition
    2.2 (i)].
  \end{enumerate}
\end{prop}
\marginproof{cor:eta_2equiv}{prop:WN-adj}

\begin{rmk}\label{rmk:about-WNC}
  For a strict 2-groupoid $\cC$, the 1-cells of $WN\cC$ are freely
  generated by the underlying graph of the 1-cells of $\cC$
  \cite[Display (2.10)]{MS93Algebraic}.
\end{rmk}

Now the nerve functor is defined on all of $\iicatnop$, not just the
subcategory of 2-groupoids or strict 2-groupoids.  Thus we can define
a functor $ \iicatnop \to \siigpd$ using the composite $W \circ
N$. Perhaps surprisingly, this composite is also a left adjoint even
though $N$ is a right adjoint.

\begin{prop}\label{prop:WN-adj}
  The functor $WN$ extends to a functor $\iicatnop \to \siigpd$, and
  is left adjoint to the inclusion
  $i \cn \siigpd \hookrightarrow \iicatnop$.
\end{prop}
\begin{proof}\proofof{prop:WN-adj}\marginref{thm:W}
  First note that $N$ extends to a full and faithful functor
  $\iicatnop \to \sset$ by \cite{Gur2009Nerves}.  Thus we have natural
  isomorphisms
  \[
    \siigpd(WN\cA, \cB) \cong \sset(N\cA, N\cB) \cong \iicatnop(\cA, i\cB).
  \]
\end{proof}

\begin{rmk}\label{rmk:nop=np}
  Note that, in the above proof, $\cB$ is a strict 2-groupoid so in
  particular every normal oplax functor $\cA \to i\cB$ is in fact a normal
  pseudofunctor.  Thus we have a natural isomorphism
  $\siigpd(WN\cA, \cB) \cong \iicatnps(\cA, i\cB)$ as well, so $WN$ is also
  left adjoint to the inclusion $\siigpd \hookrightarrow \iicatnps$.
\end{rmk}

Note since the nerve functor $N\cn \iicatnop \to \sset$ is full and
faithful, $\eta_{N\cK}$ is in fact in the image of $N$.

\begin{notn}\label{notn:epdot}
  Let $\epz_{\cK}^{\centerdot} \cn \cK \to WN \cK$ be the unique
  normal pseudofunctor (\cref{rmk:nop=np}) with
  $N\epz^{\centerdot}_{\cK} = \eta_{N \cK}$. Note that
  $\epz^{\centerdot}$ is strictly natural in normal pseudofunctors.
\end{notn}
\marginproof{notn:epdot}{cor:eta_2equiv}

\begin{rmk}\label{rmk:epz-is-retraction}
  The triangle identities for $\eta$ and $\epz$ show that the
  composite $\epz \epz^\centerdot$ is the identity 2-functor. The unit
  and counit of the adjunction $WN \dashv i$ are given, respectively,
  by $\epz^\centerdot$ and $\epz$.
\end{rmk}

\begin{lem}\label{lem:epz-P2-equiv}
  The transformations $\epz^\centerdot$ and $\epz$ are $P_2$-equivalences.
\end{lem}
\begin{proof}\proofof{lem:epz-P2-equiv}
  Recall that $\eta$ is a $P_2$-equivalence by
  \cref{cor:eta_2equiv}~(\ref{it:eta}).  Since the nerve functor
  creates $P_2$-equivalences of 2-categories, $\epz^\centerdot$ is
  also a $P_2$-equivalence.  This implies that $\epz$ is a
  $P_2$-equivalence by \cref{rmk:epz-is-retraction} and 2-out-of-3.
\end{proof}

This accomplishes the first goal of this section, to produce from a
2-category $\cC$ a strict 2-groupoid $WN \cC$ and a pseudofunctor $\cC \to
WN \cC$ which is a natural $P_2$-equivalence.  We now turn to
incorporating the symmetric monoidal structure.

\begin{prop}\label{prop:WN-monoidal-adjunction}
  The adjunction $WN \dashv i$ of \cref{prop:WN-adj} is monoidal with
  respect to the cartesian product.
\end{prop}
\begin{proof}\proofof{prop:WN-monoidal-adjunction}
  Note that $N$ preserves products since it is a right adjoint
  (\cref{prop:classical-monoidal-functors}), and therefore $WN$ is
  strong monoidal by \cref{cor:eta_2equiv}~(\ref{it:W-monoidal}).
  Straightforward calculations show that $\epz^\centerdot$ and $\epz$ are
  monoidal transformations.  
\end{proof}

\begin{notn}\label{notn:maps-to-WSBO}
  Let
  \[
  h \cn WN(-) \Rightarrow WS|N(-)| = WSB(-)
  \]
  denote the natural transformation induced by the unit $\id
  \Rightarrow S|-|$.  
\end{notn}

Applying symmetric monoidal functors to the operad $\cO$, we have the
following corollary (see \cref{lem:operad_monoidal}).
\begin{cor}\label{cor:operads-WNO-WSBO}
  There are operads $WN\cO$ and $WSB\cO$.  The transformations
  $\epz$, $\epz^\centerdot$, and $h$ induce operad maps
  \[
  \epz \cn WN\cO \to \cO, \qquad 
  \epz^\centerdot \cn \cO \to WN\cO, \qquad \text{and} \qquad 
  h \cn WN\cO \to WSB\cO.
  \]
\end{cor}
\marginproof{cor:operads-WNO-WSBO}{prop:WN-monoidal-adjunction,defn:operadO,lem:operad_monoidal}

\begin{notn}\label{defn:wt-can-choice}
Let $\wncan$ denote the choice of multiplication in $WN\cO$ given
by
  applying $\epz^\centerdot$ to the canonical choice $\can$
  (\cref{prop:operad-map-induces-mult-choice}).
\end{notn}
 
\begin{prop}\label{prop:zigzag1}
  Given a permutative Gray-monoid $\cA$, there is a natural zigzag of
  strict functors of symmetric monoidal 2-categories as shown
  below. The left leg is a $P_2$-equivalence and the right leg is a
  biequivalence.
     \[
  \begin{tikzpicture}[x=30mm,y=20mm]
    \draw[tikzob,mm] 
    (0,0) node (a) {\cA}
    ++(.35,0) node (ka) {\can^* \cA}
    ++(.55,0) node (ka') {\wncan^* \epz^* \cA}
    ++(1,.8) node (wna) {\wncan^* WN(\cA)}
    ++(1,-.8) node (wsba) {\wncan^* h^* WSB(\cA).}
    ;
    \path[tikzar,mm] 
    (wna) edge[swap] node {\wncan^*(\epz_{\cA})} (ka')
    (wna) edge node {\wncan^*(h_{\cA})} (wsba)
    ;
    \node[mm] at ($ (a) !.4! (ka) $) {=};
    \node[mm] at ($ (ka) !.45! (ka') $) {=};
  \end{tikzpicture}
  \] 
\end{prop}
\begin{proof}\proofof{prop:zigzag1}\marginref{thm:choice-of-mult-to-SMB,cor:operads-WNO-WSBO}
  Recall that we implicitly regard $\cA$ as an
  $\cO$-algebra via \cref{prop:PGM-is-OAlg}. Therefore we
  have a zigzag of $WN\cO$-algebra maps (note that these have underlying
 2-functors) induced by the
  components of $\epz$ and $h$, respectively,
  \[
    \epz^* \cA \longleftarrow WN(\cA) \longrightarrow 
    h^* WSB(\cA).
  \]
We have $\cA = \can^* \cA$ by
  \cref{prop:PGM-as-O-algebra}.  Note $\wncan^* \epz^* = \can^*$
  because $\epz \epz^\centerdot = \id$ (\cref{rmk:epz-is-retraction}).
  This gives a zigzag of symmetric monoidal 2-categories and strict functors.
  Naturality follows from naturality of
  $\epz$ and $h$.  Moreover, $\epz$ is a $P_2$-equivalence by
  \cref{lem:epz-P2-equiv} and $h$ is a biequivalence
  because $W$ sends weak equivalences to
  biequivalences by \cref{cor:eta_2equiv}~(\ref{it:W-preserves}).
\end{proof}

It is clear that the property of being Picard is preserved by
biequivalences and, moreover, every $P_2$-equivalence of Picard
2-categories is a biequivalence.
Therefore we have the following corollary of \cref{prop:zigzag1}.
\begin{cor}\label{cor:zigzag1-picard}
  If $\cA$ is a strict Picard 2-category, then the span in
  \cref{prop:zigzag1} is a span of Picard 2-categories.
  \marginproof{cor:zigzag1-picard}{prop:zigzag1}
\end{cor}

\subsection{\texorpdfstring{$E_\infty$}{E\_infty}-algebras and stable homotopy theory of
  symmetric monoidal bicategories}
\label{sec:K-thy}

In this section, we show that the composite $WS$, combined with
any choice of multiplication, sends stable equivalences of $E_\infty$
spaces to stable $P_2$-equivalences of symmetric monoidal 2-groupoids.

Our notions of stable equivalence, stable $n$-equivalence, and
$P_n$-equivalence for strict functors of symmetric monoidal bicategories
are created by the $K$-theory functors of
\cite{GJO2017KTheory,GO2012Infinite}, which construct infinite loop
spaces from bicategories and 2-categories.  We begin with a review of
these functors and then apply the theory of $E_\infty$ algebras in
$\Top$.

\begin{thm}[{\cite{GO2012Infinite,GJO2017KTheory}}]\label{thm:property-of-Kthy}
  There is a functor $K\cn \SMBicat_{s} \to \Sp_{0}$.
  For a symmetric monoidal bicategory $\cA$, $K\cA$ is a positive
  $\Om$-spectrum, with the property that
  \[
  B\cA \simeq K\cA(0) \to \Om K\cA(1)
  \]
  is a group-completion. In particular, we have that 
  \[
  \pi_n(K\cA) \cong \pi_n(\Om B (B\cA)),
  \]
  where the latter are the unstable homotopy groups of the topological
  group-completion of the classifying space $B\cA$.  
\end{thm}

\begin{defn}\label{defn:st-eq-SMB}
  A strict functor $F\cn \cA \to \cB$ of symmetric monoidal bicategories is a
  \emph{stable equivalence} if the induced map of spectra $KF\cn K\cA
  \to K\cB$ is a stable equivalence. Similarly, $F$ is a stable
  $n$-equivalence, respectively stable $P_n$-equivalence, if $KF$ is so.
\end{defn}

\begin{lem}
  Let $F \cn \cA \to \cB$ be a strict functor such that
  $BF \cn B\cA \to B\cB$ is a weak equivalence. Then $F$ is a stable
  equivalence, and hence, also a stable $P_n$-equivalence for all
  $n\geq 0$.
\end{lem}
\begin{proof}
 The corresponding map of spectra $KF$ is a level equivalence.
\end{proof}

Restricting to permutative Gray-monoids, we obtain the main result in
\cite{GJO2017KTheory}.

\begin{thm}[\cite{GJO2017KTheory}]
  There is a functor $K\cn \PGM \to \Sp_{0}$ which induces an
  equivalence of homotopy theories
  \[
    (\PGM,\steq) \hty (\Sp_{0}, \steq)
  \]
  between permutative Gray-monoids and connective spectra, working
  relative to the stable equivalences.
\end{thm}

\begin{prop}[{\cite[Remark 6.4]{GJO2017KTheory}}]
  When restricted to the subcategory $\PGM$, the functor $K$ of
  \cite{GO2012Infinite} is equivalent to that of \cite{GJO2017KTheory}.
\end{prop}

\begin{defn}
  An operad $\sD$ in $\Top$ is an \emph{$E_\infty$ operad} if for all
  $n\geq 0$, the $\Si_n$-action on $\sD(n)$ is free, and $\sD(n)$ is
  contractible.
\end{defn}

The following theorem appeared first in \cite{May1974Einfty}.  A modern
(equivariant) version is in \cite{GM2017Iterated}.
\begin{thm}[{\cite[Theorem 2.3]{May1974Einfty}}, {\cite[Theorem 1.14, Definition 2.7]{GM2017Iterated}}]
  \label{thm:spectra-from-Einfty}
  Let $\sD$ be an $E_\infty$ operad in $\Top$. There is a functor
  \[
  \bE\cn  \sD\mh\mb{Alg} \to \Spectra
  \]
  such for a $\sD$-algebra $X$ and all $n\geq 2$,
  \[
  X=\bE(X)(0)\to \Om^n \bE(X)(n)
  \]
  is a group-completion.
\end{thm}

\begin{defn}
  Let $\sD$ be an $E_\infty$ operad in $\Top$. A map $f\cn X \to Y$ of
  $\sD$-algebras is a \emph{stable equivalence} if the associated map
  $\bE(f)$ of spectra is so. Similarly, $f$ is said to be a \emph{stable
  $P_n$-equivalence} if $\bE(f)$ is so.
\end{defn}

We use \cref{thm:spectra-from-Einfty} to recognize stable equivalences
and stable $P_n$-equivalences of $\sD$-algebras by their induced maps
on group-completions, as in the following result.
\begin{cor}\label{lem:steq-Einfty-gp-cpn}
  A map $f\cn X \to Y$ of $\sD$-algebras is a stable equivalence if
  and only if the associated map on group-completions
 \[
 \Om B f \cn \Om B X \to \Om BY
 \]
 is an unstable equivalence.
 
 Similarly, $f$ is a stable $P_n$-equivalence if and only if
 \[
 \Om B f \cn \Om B X \to \Om BY
 \]
 is an unstable $P_n$-equivalence.
 \marginproof{lem:steq-Einfty-gp-cpn}{thm:spectra-from-Einfty}
\end{cor}

Applying \cref{thm:property-of-Kthy}, we can recognize stable
$P_n$-equivalences of symmetric monoidal bicategories in the same way.
\begin{cor}\label{prop:P_n-equiv-gp-cpn}
  \marginproof{prop:P_n-equiv-gp-cpn}{thm:property-of-Kthy} 
  A strict functor $F \cn \cA \to \cB$ of symmetric monoidal bicategories is a
  stable $P_n$-equivalence if and only if the associated map on
  topological group-completions
  \[
  \Om B (BF)\cn \Om B(B\cA) \to \Om B(B\cB)
  \]
  is an unstable $P_n$-equivalence.
\end{cor}

\begin{prop}\label{lem:WS-preserves-st2eq}
  Let $\sD$ be an $E_\infty$ operad, and let $\chi$ denote any choice
  of multiplication for $WS(\sD)$.  If $\al \colon X\to Y$ is a map of
  $\sD$-algebras which is a stable equivalence, then $\chi^* WS \al
  \colon \chi^*WS X \to \chi^*WS Y$ is a stable $P_2$-equivalence in
  $\SMIICat_{s}$.
\end{prop}
\begin{proof}\proofof{lem:WS-preserves-st2eq}
  By \cref{prop:P_n-equiv-gp-cpn} and \cref{lem:steq-Einfty-gp-cpn},
  it suffices to show that $BWS\al$ is a stable $P_2$-equivalence.
  Consider the following diagram of algebras over $|S\sD|$,
  induced by naturality of the counit
  \[
  |S(-)| \Rightarrow \id
  \]
  and the transformation
  \[
  |S(-)| \Rightarrow BWS(-)
  \]
  induced by the unit of $W \dashv N$.

  \[
  \begin{tikzpicture}[x=30mm,y=15mm]
    \draw[tikzob,mm] 
    (0,2) node (X) {X}
    (1,2) node (Y) {Y}
    (0,1) node (SX) {|SX|}
    (1,1) node (SY) {|SY|}
    (0,0) node (BWSX) {BWSX}
    (1,0) node (BWSY) {BWSY}
    ;
    \path[tikzar,mm] 
    (X) edge node {\al} (Y)
    (SX) edge node {|S\al|} (SY)
    (BWSX) edge node {BWS\al} (BWSY)
    (SX) edge node {} (X)
    (SX) edge node {} (BWSX)
    (SY) edge node {} (Y)
    (SY) edge node {} (BWSY)
    ;
  \end{tikzpicture}
  \]
  The upper vertical arrows are unstable weak equivalences, therefore
  stable equivalences.  The lower vertical arrows are unstable
  $P_2$-equivalences by \cref{cor:eta_2equiv}~(\ref{it:eta}).  Since
  $W$ takes values in 2-groupoids, $BWSX$ and $BWSY$ are
  3-coconnected.  Therefore by \cref{cor:n-equiv-alt} and
  \cref{lem:steq-Einfty-gp-cpn} the lower vertical morphisms are stable
  $P_2$-equivalences. The assumption that $\al$ is a stable
  equivalence means that $|S\al|$ must be too, and hence both are
  stable $P_2$-equivalences.  The result then follows by 2-out-of-3
  for stable $P_2$-equivalences.
\end{proof}

\subsection{Group-completion for \texorpdfstring{$E_\infty$}{E\_infty} algebras} 
\label{sec:top-gp-cpn} 

In this section we recall the theory of group-completions of $E_\infty$ algebras
in $\Top$ and discuss its implications for the symmetric monoidal
2-groupoids studied above.  Let $\sD$ be an arbitrary $E_\infty$
operad in $\Top$. 

\begin{notn}\label{notn:Cinfty-Dinfty}
  Let $\sC_n$ be the little $n$-cubes operad, and let $\sC_\infty$ be
  the colimit (the maps are given by inclusions of $\sC_n$ into
  $\sC_{n+1}$). This is an $E_\infty$ operad (see \cite[Section
  4]{may72geo}). Let $\sD_\infty = \sD \times \sC_\infty$ and let
  $p_1$ and $p_2$ denote the two projections.
\end{notn}
 
\begin{thm}[{\cite[Theorem 2.3]{May1974Einfty}}]\label{gpcomp}
  If $X$ is a $\sD$-algebra, then there is an algebra $qX$ over
  $\sD_\infty$ and a $\sC_\infty$-algebra $LX$, together with
  $\sD_\infty$-algebra maps
  \[
  \begin{tikzpicture}[x=15mm,y=20mm]
    \draw[tikzob,mm] 
    (-1,0) node (L) {p_1^* X}
    (0,0) node (M) {qX}
    (1,0) node (R) {p_2^* LX}    
    ;
    \path[tikzar,mm] 
    (M) edge[swap] node {\xi} (L)
    (M) edge node {\al} (R)
    ;
  \end{tikzpicture}
  \]
  such that $\xi$ is a homotopy equivalence and $\al$ is a
  group-completion. The assignments $X \mapsto qX$ and $X \mapsto LX$
  are functorial, and $\xi$ and $\al$ are natural.
\end{thm}
\begin{rmk}
The functors $q$ and $L$ are constructed explicitly in
\cite{May1974Einfty}. The homotopy inverse of $\xi$ is also very
explicit, but it is not a
  $\sD_\infty$-algebra map. 
\end{rmk}

Note that \cref{lem:steq-Einfty-gp-cpn} implies that both $\xi$ and
$\al$ above are stable equivalences of $\sD_\infty$-algebras.  We now
specialize to the $E_\infty$ operad $B\cO$, and we let
$B\cO_\infty=B\cO\times\sC_\infty$. Since the functors $W$ and $S$ are
strong symmetric monoidal we obtain the following result.

\begin{lem}\label{cor:operads-WSC-WSD}
There are operads
$WS\sC_\infty$ and
$WSB\cO_\infty$ in $\iicat$, together with projections
\[
p_1\cn WSB\cO_\infty \to WSB\cO
\]
and
\[
p_2\cn WSB\cO_\infty \to WS\sC_\infty.
\]
\end{lem}
\marginproof{cor:operads-WSC-WSD}{cor:eta_2equiv,prop:classical-monoidal-functors}

\begin{cor}\label{cor:WSD-zigzag}
Given a permutative Gray-monoid $\cA$, there exists a natural zigzag
of maps
  of $WSB\cO_\infty$-algebras in $(\iicat,\times)$ 
  \[
  \begin{tikzpicture}[x=30mm,y=20mm]
    \draw[tikzob,mm] 
    (0,0) node (wsba') {p_1^* WS(B\cA)}
    ++(1,.8) node (wsqba) {WS(qB\cA) }
    ++(1,-.8) node (wslba) {p_2^* WS(LB\cA).}
    ;
    \path[tikzar,mm] 
    (wsqba) edge[swap] node {WS(\xi_{B\cA})} (wsba')
    (wsqba) edge node {WS(\al_{B\cA})} (wslba)
    ;
  \end{tikzpicture}
  \]
  The left arrow is a biequivalence.
\end{cor}
\begin{proof}\proofof{cor:WSD-zigzag}\marginref{cor:operads-WSC-WSD}
  By
  \cref{prop:classical-monoidal-functors,cor:eta_2equiv}~(\ref{it:W-monoidal}),
  $B$, $S$ and $W$ are strong monoidal.  For the biequivalence part, $S$
  sends homotopy equivalences to weak equivalences, and $W$ sends weak
  equivalences to biequivalences (\cref{cor:eta_2equiv}~(\ref{it:W-preserves})).
\end{proof}

\begin{notn}\label{prop:WSCinfty-mult-choice}
  Because $WS\sC_\infty$ is in $\bC(\leq\infty)$,
  \cref{prop:C4-choice-of-mult} guarantees that it has a choice of
  multiplication.  For the rest of this paper, let $\nu$ denote a
  fixed such choice.
  For example, the content of \cite[\S 2.2]{GO2012Infinite} provides one
  such choice.
\end{notn}

\begin{notn}\label{WSDinfty-mult-choice}
  Let $\CAN$ denote the choice of multiplication in $WSB\cO_\infty$
  given by the product of $h(\wncan)$
 (\cref{notn:maps-to-WSBO,defn:wt-can-choice}) and $\nu$.
We  call this the canonical choice of multiplication for $WSB\cO_\infty$.
\marginproof{WSDinfty-mult-choice}{prop:choice-from-product,rmk:product-choice-projections}  
\end{notn}

\begin{prop}\label{prop:WSLBA-picard}
  Given a permutative Gray-monoid $\cA$, there is a natural zigzag of
  symmetric monoidal 2-categories and strict monoidal 2-functors
  \[
  \begin{tikzpicture}[x=30mm,y=20mm]
    \draw[tikzob,mm] 
    (0,0) node (wsba) {\wncan^* h^* WS(B\cA)}
    ++(.94,0) node (wsba') {\CAN^* p_1^* WS(B\cA)}
    ++(1,.8) node (wsqba) {\CAN^* WS(qB\cA) }
    ++(1,-.8) node (wslba) {\CAN^* p_2^* WS(LB\cA)}
    ++(1,0) node (wslba') {\nu^* WS(LB\cA).}
    ;
    \path[tikzar,mm] 
    (wsqba) edge[swap] node {\CAN^*(WS(\xi_{B\cA}))} (wsba')
    (wsqba) edge node {\CAN^*(WS(\al_{B\cA}))} (wslba)
    ;
    \node[mm] at ($ (wslba) !.55! (wslba') $) {=};
    \node[mm] at ($ (wsba) !.50! (wsba') $) {=};
  \end{tikzpicture}
  \]
  Moreover, $\nu^* WS(LB\cA)$ is a Picard 2-category, $\xi$ is a
  biequivalence and $\al$ is a stable $P_2$-equivalence.
\end{prop}
\begin{proof}\proofof{prop:WSLBA-picard}\marginref{thm:choice-of-mult-to-SMB}
  The existence of this natural zigzag follows by applying
  \cref{cor:WSD-zigzag} with the canonical choice of multiplication
  $\CAN$.  By \cref{prop:chistar-fstar-equals-chifstar} we identify
  $\CAN^*p_1^* = \wncan^*h^*$ and $\CAN^*p_2^* = \nu^*$.

  We see that $\nu^* WS(LB\cA)$ is a 2-groupoid because $WX$ is a
  2-groupoid for every simplicial set $X$ (see \cref{thm:W}). We note
  that $\nu^*WS$ preserves $\pi_0$ and that
  $\pi_0(\nu^* WS(LB\cA))\cong \pi_0(LB\cA)$ is a group because
  $LB\cA$ is group-complete.  The
  product is induced by the monoidal structure, and therefore it follows
  that objects have inverses up to equivalence.  Thus
  $\nu^* WS(LB\cA)$ is a Picard 2-category.

  The fact that $\CAN^* WS(\xi)$ is a biequivalence is immediate from
  \cref{cor:WSD-zigzag}, and the claim about $\CAN^* WS(\al)$ follows
  from \cref{lem:WS-preserves-st2eq} because the
  map $\al$ of \cref{gpcomp} is a group-completion and hence a stable equivalence.
\end{proof}

Because the property of being Picard is preserved by biequivalences, we
have the following corollary.
\begin{cor}\label{cor:zigzag2-picard}
  If $\cA$ is a strict Picard 2-category, then the span in
  \cref{prop:WSLBA-picard} is a span of Picard 2-categories.
   \marginproof{cor:zigzag2-picard}{prop:WSLBA-picard}
\end{cor}

\section{Proof of the 2-dimensional stable homotopy hypothesis}
\label{sec:main}

Our main theorem is the following.  
\begin{thm}\label{thm:main}
  There is an equivalence of homotopy theories
  \[
  \htythy{\PicPGM}{\cateq} \hty \htythy{\Sp^2_{0}}{\steq}.
  \]
\end{thm}
\begin{proof}\proofof{thm:main}
  The proof follows from putting together several results in this
  section. To be precise, we combine
  \cref{prop:PicPGMcateq-equiv-PicPGMst2eq,prop:PGMst2eq-equiv-Ga2Catst2eq}
  below,
  which follow easily from previous work in
  \cite{GJO2017KTheory,GJOS2017Postnikov}, with \cref{thm:linchpin},
  whose proof depends on the content of \cref{sec:algebra,sec:operads,sec:gp-cpn}.
\end{proof}

\begin{prop}\label{prop:PicPGMcateq-equiv-PicPGMst2eq}
There is an equality of homotopy theories
  \[
  \htythy{\PicPGM}{\cateq} = \htythy{\PicPGM}{\stpiieq}.
  \]
\end{prop}
\begin{proof}\proofof{prop:PicPGMcateq-equiv-PicPGMst2eq}
  Recall that a 2-functor is a biequivalence if and only if it is
  essentially surjective and a local equivalence.  The formulas of
  \cite[Lemma 3.2]{GJOS2017Postnikov} show that the stable homotopy
  groups of a strict Picard 2-category are computed by the
  algebraic homotopy groups (i.e.\ equivalence classes of invertible
  morphisms) in each dimension.  Therefore a strict functor between
 strict Picard
  2-categories is a stable $P_2$-equivalence if and only
  if it is a biequivalence.
\end{proof}

\begin{lem}\label{lem:P-preserves-P2}
  The functor $P$ of \cite{GJO2017KTheory} preserves stable
  $P_2$-equivalences.
\end{lem}
\begin{proof}\proofof{lem:P-preserves-P2}
  Using the notation of \cite{GJO2017KTheory}, let $f \cn X \to Y$ be a stable $P_2$-equivalence of
  $\Ga$-2-categories (certain diagrams of 2-categories indexed on
  finite pointed sets; see \cite[Definition 2.11]{GJO2017KTheory}). Since stable
  $P_2$-equivalences of permutative Gray-monoids are created by the $K$-theory functor of
  \cite{GJO2017KTheory}, it suffices to check that $KPf$ is a stable
  $P_2$-equivalence. This is immediate from the naturality of the unit
  $\eta$ with respect to strict $\Ga$-maps (\cite[Corollary
  7.14]{GJO2017KTheory}): we have
  \[
    \eta \circ f = KPf \circ \eta. 
  \] 
  Since $\eta$ is a stable equivalence, then $KPf$ is a stable
  $P_2$-equivalence by 2-out-of-3, and therefore $Pf$ is too.
\end{proof}
\begin{prop}\label{prop:PGMst2eq-equiv-Ga2Catst2eq}
There are equivalences of homotopy theories
  \[
    \htythy{\PGM}{\stpiieq} \hty \htythy{\Sp_{\ge 0}}{\stpiieq}\hty
    \htythy{\Sp^2_{0}}{\steq}.
  \]
\end{prop}
\begin{proof}\proofof{prop:PGMst2eq-equiv-Ga2Catst2eq}
  The $K$-theory functor of \cite[Proposition 6.13]{GJO2017KTheory}
  creates stable $P_2$-equivalences by definition.
  \Cref{lem:P-preserves-P2} observes that the inverse $P$ preserves
  stable $P_2$-equivalences as well.  The first equivalence then
  follows from the equivalences of \cite{GJO2017KTheory} relative to
  stable $P_2$-equivalences.  The second equivalence is a
  reformulation of definitions.
\end{proof}

\begin{thm}\label{thm:linchpin}
  There is an equivalence of homotopy theories
  \[
  \htythy{\PicPGM}{\stpiieq} \hty \htythy{\PGM}{\stpiieq}.
  \]
\end{thm}
To prove \cref{thm:linchpin}, we consider the serially-commuting diagram of
homotopy theories and relative functors below.  \cref{lem:PGM-to-SMB-equiv}
shows that the inclusions $j$ in this diagram are equivalences of
homotopy theories, with inverse equivalences given by $r = (-)^{qst}$.
We will show that the inclusions $i$ are equivalences of homotopy
theories.
\begin{align}\label{align:linchpin-diagram}
\begin{tikzpicture}[x=60mm,y=20mm,baseline={(0,10mm)}]
  \draw[tikzob,mm] 
  (0,0) node (psmb) {\htythy{\Pic\Bicat_{s}}{\stpiieq}}
  (1,0) node (smb) {\htythy{\SMBicat_{s}}{\stpiieq}}
  (0,1) node (ppgm) {\htythy{\PicPGM}{\stpiieq}}
  (1,1) node (pgm) {\htythy{\PGM}{\stpiieq}}
  ;
  \path[tikzar,mm] 
  (psmb) edge[swap] node {i} (smb)
  (ppgm) edge node {i} (pgm)
  (ppgm) edge[bend right, swap] node {j} (psmb) 
  (pgm) edge[bend right, swap] node {j} (smb) 
  (psmb) edge[bend right, swap] node{r} (ppgm)
  (smb) edge[bend right, swap] node{r} (pgm)
  (pgm) edge[dashed, swap] node {\bG} (psmb)
  ;
\end{tikzpicture}
\end{align}
To do this, we first reduce to the problem of constructing a relative
functor $\bG$ which commutes with $i$ and $j$ up to natural zigzags of
stable $P_2$-equivalences.
\begin{lem}\label{lem:G-suffices}
  Suppose there is a relative functor $\bG$ as shown in
  \cref{align:linchpin-diagram}, and suppose that diagram involving
  $\bG$, $i$, and $j$ commutes up to a natural zigzag of stable
  $P_2$-equivalences.  Then the inclusions labeled $i$ are
  equivalences of homotopy theories.
\end{lem}

\begin{proof}\proofof{lem:G-suffices}
  Because the square involving $i$ and $j$ commutes, it suffices to
  prove that the inclusion
  \[
    i \cn   \htythy{\Pic\Bicat_{s}}{\stpiieq} \to \htythy{\SMBicat_{s}}{\stpiieq}
  \]
  is an equivalence of homotopy theories.  We do this by showing that
  the composite $\bG r$ is an inverse for $i$ up to natural zigzag of
  stable $P_2$-equivalences.

  \newcommand{\zzag}{\stackrel{\sim}{\leftrightarrow\leftrightarrow}}
  Let us write $\zzag$ to denote a natural zigzag of stable
  $P_2$-equivalences.  Then the proof of \cref{lem:PGM-to-SMB-equiv}
  shows we have
  \[
    jr \zzag \id \qquad \mathrm{ and } \qquad rj \zzag \id.
  \]
  By assumption, we have $i\bG \zzag j$ and $\bG i \zzag j$.  Hence we have
  \[
    i\bG r \zzag jr \zzag \id \qquad \mathrm{ and } 
    \qquad \bG ri = \bG ir \zzag jr \zzag \id.
  \]
\end{proof}

Now we describe $\bG$ and show that it satisfies the hypotheses of \cref{lem:G-suffices}.
\begin{defn}\label{defn:G}
  Let $\bG=\nu^* WS(LB-)$.
\end{defn}
Recalling the relevant notation, this is the composite of the
classifying space $B$, topological group completion $L$, singular
simplicial set $S$, Whitehead 2-groupoid $W$, and choice of
multiplication $\nu^{*}$ (applied to a permutative Gray-monoid
considered as an $\cO$-algebra via \cref{prop:PGM-is-OAlg}).  By
\cref{prop:WSLBA-picard}, this is a functor from permutative
Gray-monoids to Picard 2-categories.  We will confirm that $\bG$ is a
relative functor in the course of the proof of \cref{thm:linchpin}.

\begin{rmk}
  The attentive reader will note that $\nu^*$ takes values in the
  subcategory $\Pic\iicat_{s} \subset \Pic\Bicat_{s}$. We implicitly
  compose with this inclusion because, although we suspect $\PGM$,
  $\SMIICat_{s}$, and $\SMBicat_{s}$ all have equivalent homotopy
  theories (and likewise for the Picard subcategories of each), the
  proof of \cref{lem:PGM-to-SMB-equiv} does not specialize to
  $\SMIICat_{s}$.
\end{rmk}

\begin{proof}[Proof of \cref{thm:linchpin}]\proofof{thm:linchpin}
  The necessary zigzags to apply \cref{lem:G-suffices} have already
  been constructed; we review them now.  Let $\cA$ be a permutative
  Gray-monoid. To compare $j\cA$ and $i\bG(\cA)$, we require three
  operads: $B\cO$ is the geometric realization of categorical
  Barratt-Eccles operad $\cO$ (\cref{defn:operadO}); $\sC_\infty$ is
  the little infinite cubes operad (\cref{notn:Cinfty-Dinfty}); and
  $B\cO_\infty = B\cO \times \sC_\infty$ is their product.

  We consider choices of multiplication induced by operad maps
  shown in \cref{disp:operad-zigzag-2} below (see
  \cref{cor:operads-WNO-WSBO,cor:operads-WSC-WSD}).
  \begin{align}
    \begin{tikzpicture}[x=23mm,y=20mm,baseline=(current bounding box.center)]
      \usetikzlibrary{calc}
      \draw[tikzob,mm] 
      (0,0) node (O) {\cO}
      (1,.8) node (WNO) {WN\cO}
      (2,0) node (WSBO) {WSB\cO}
      (3,.8) node (WSD) {WSB\cO_\infty}
      (4,0) node (WSC) {WS\sC_\infty};
      \path[tikzar,mm] 
      (WNO) edge[swap] node {\epz} (O)
      (WNO) edge node {h} (WSBO)
      (WNO) edge node {} (WSBO)
      (WSD) edge[swap] node {p_1} (WSBO)
      (WSD) edge node {p_2} (WSC);
    \end{tikzpicture}
    \label{disp:operad-zigzag-2}
  \end{align}
  This is a diagram of operads in $\iicat$, that is, at level $n$ the
  maps are given by 2-functors.
  With appropriate choices of multiplication, we construct the
  required zigzag in two stages. First, we use $\wncan$, given by
  applying $\epz^\centerdot$ to the canonical choice $\can$ (see 
  \cref{defn:wt-can-choice}).  By
  \cref{prop:zigzag1} we have the following zigzag of 
  in $\SMIICat_{s}$, where the right leg is a biequivalence and the
  left leg is an unstable $P_2$-equivalence and therefore a stable
  $P_2$-equivalence by \cref{cor:n-equiv-alt,lem:steq-Einfty-gp-cpn}.
  \[
    \begin{tikzpicture}[x=30mm,y=20mm]
      \draw[tikzob,mm] 
      (0,0) node (a) {\cA}
      ++(.35,0) node (ka) {\can^* \cA}
      ++(.55,0) node (ka') {\wncan^* \epz^* \cA}
      ++(1,.8) node (wna) {\wncan^* WN(\cA)}
      ++(1,-.8) node (wsba) {\wncan^* h^* WSB(\cA)}
      ;
      \path[tikzar,mm] 
      (wna) edge[swap] node {\wncan^*(\epz_{\cA})} (ka')
      (wna) edge node {\wncan^*(h_{\cA})} (wsba)
      ;
      \node[mm] at ($ (a) !.4! (ka) $) {=};
      \node[mm] at ($ (ka) !.45! (ka') $) {=};
    \end{tikzpicture}
  \]
  Second, we use $\CAN$, described in \cref{WSDinfty-mult-choice}.
  By \cref{prop:WSLBA-picard} we have the following zigzag in
  $\SMIICat_{s}$, where the left leg is a biequivalence, and the
  right leg is a stable $P_2$-equivalence.
  \[
    \begin{tikzpicture}[x=30mm,y=20mm]
      \draw[tikzob,mm] 
      (0,0) node (wsba) {\wncan^* h^* WS(B\cA)}
      ++(.94,0) node (wsba') {\CAN^* p_1^* WS(B\cA)}
      ++(1,.8) node (wsqba) {\CAN^* WS(qB\cA) }
      ++(1,-.8) node (wslba) {\CAN^* p_2^* WS(LB\cA)}
      ++(1,0) node (wslba') {\nu^* WS(LB\cA)}
      ;
      \path[tikzar,mm] 
      (wsqba) edge[swap] node {\CAN^*(WS(\xi_{B\cA}))} (wsba')
      (wsqba) edge node {\CAN^*(WS(\al_{B\cA}))} (wslba)
      ;
      \node[mm] at ($ (wslba) !.55! (wslba') $) {=};
      \node[mm] at ($ (wsba) !.50! (wsba') $) {=};
    \end{tikzpicture}
  \]
  Thus we have a natural zigzag of stable $P_2$-equivalences between
  $i\bG$ and $j$.  This also shows that $\bG$ is a relative functor
  since $j$ preserves and $i$ creates stable $P_2$-equivalences.

  As noted in
  \cref{cor:zigzag1-picard,cor:zigzag2-picard}, this is a zigzag of
  Picard 2-categories when $\cA$ is a strict Picard 2-category.  Thus
  we also have a natural zigzag of stable $P_2$-equivalences between
  $\bG i$ and $j$.  By \cref{lem:G-suffices}, this completes the proof.
\end{proof}

The key step, producing a zigzag of stable $P_2$-equivalences between
$i\bG$ and $j$, is summarized in \cref{fig:summary}.
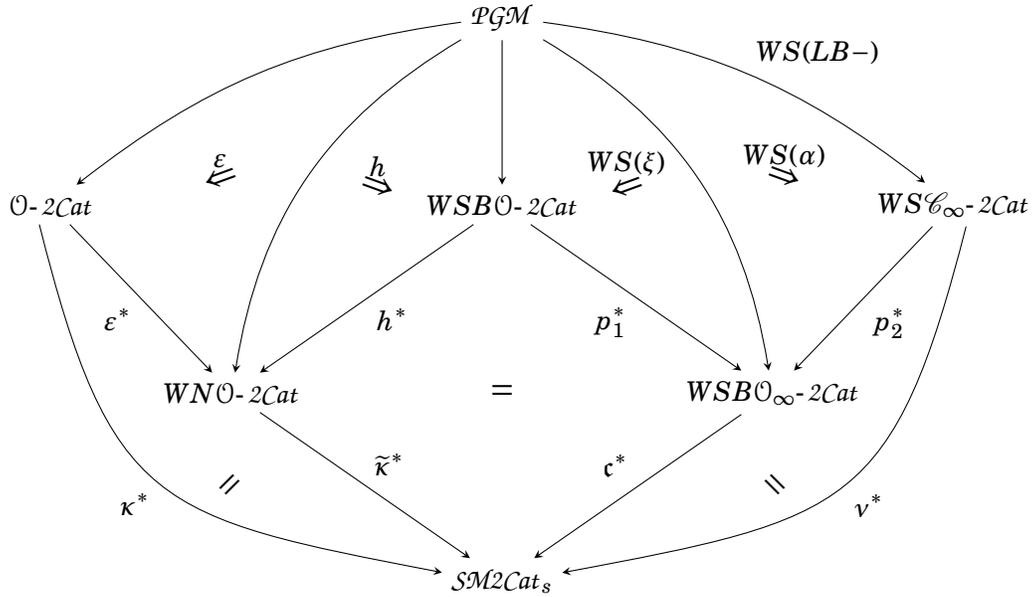
\begin{figure}[H]
\begin{tikzpicture}[x=30mm,y=25mm]
  \draw[tikzob,mm] 
  (0,3) node (pgm) {\PGM}
  (0,2) node (wsbo) {WSB\cO\mh\iicat}
  (-2,2) node (o) {\cO\mh\iicat}
  (2,2) node (wsc) {WS\sC_\infty\mh\iicat}
  (-1.2,1) node (wno) {WN\cO\mh\iicat}
  (1.2,1) node (wsd) {WSB\cO_\infty\mh\iicat}
  (0,0) node (sm2) {\SMIICat_{s}}
  ;
  \path[->,auto,mm] 
  (pgm) edge[swap, bend right=15] node {} (o)
  (pgm) edge[swap, bend right=25] node {} (wno)
  (pgm) edge node {} (wsbo)
  (pgm) edge[bend left=25] node {} (wsd)
  (pgm) edge[bend left=15] node {WS(LB-)} (wsc)
  (o.240) edge[swap, bend right=35, in=210, looseness=1.4] node {\can^*} (sm2)
  (o) edge[swap] node {\epz^*} (wno)
  (wsbo) edge node {h^*} (wno)
  (wsbo) edge[swap] node {p_1^*} (wsd)
  (wsc) edge node {p_2^*} (wsd)
  (wsc.300) edge[bend left=35, in=150, looseness=1.4] node {\nu^*} (sm2)
  (wno) edge node {\wncan^*} (sm2)
  (wsd) edge[swap] node {\CAN^*} (sm2)
  ;
  \draw[tikzob,mm]
  (-1.25,2.15) node[rotate=200,font={\Large}] {\Rightarrow} node[above]
  {\epz}
  (-0.55,2.1) node[rotate=-20,font={\Large}] {\Rightarrow} node[above]
  {h}
  (0.55,2.1) node[rotate=200,font={\Large}] {\Rightarrow} node[above]
  {WS(\xi)}
  (1.25,2.15) node[rotate=-20,font={\Large}] {\Rightarrow} node[above]
  {WS(\al)}
  (0,1) node[font={\Large}] {=}
  (-1.2,.5) node[rotate=45,font={\Large}] {=}
  (1.2,.5) node[rotate=-45,font={\Large}] {=}
  ;
\end{tikzpicture}
\caption{This diagram of categories, functors, and natural
  transformations summarizes the zigzag constructed in the proof of
  \cref{thm:linchpin}}
\label{fig:summary}
\end{figure}
Composing with the inclusion $\SMIICat_{s} \subset \SMBicat_{s}$, the
composite along the left hand side becomes the
inclusion $j$.  Likewise, the composite $\nu^* WSLB$ along the right
hand side becomes $i \bG$. The components of $h$ and
$WS(\xi)$ are 
biequivalences; the components of $\epz$ and $WS(\al)$ are
stable $P_2$-equivalences.

\bibliographystyle{gjotex/amsalpha2}
\bibliography{gjotex/Refs}%

\end{document}